\documentclass[onefignum,onetabnum]{siamart171218}

\usepackage{amsmath}

\usepackage{amssymb}
\usepackage{mathtools}
\usepackage{graphicx}
\usepackage{xcolor}
\usepackage{mathtools}
\usepackage{todonotes}
\usepackage{bbm}
\usepackage[numbers]{natbib}
\usepackage[normalem]{ulem}

\newcommand{\map}[1]{\textcolor{red}{#1}}

\newcommand{\acc}[1]{\textcolor{brown}{#1}}

\newcommand{\vx}{\mathbf{x}}
\newcommand{\vy}{\mathbf{y}}

\newcommand{\mW}{\mathbf{W}}
\newcommand{\mD}{\mathbf{D}}
\newcommand{\mM}{\mathbf{M}}
\newcommand{\mH}{\mathbf{H}}
\newcommand{\abs}[1]{\left\lvert#1\right\rvert}
\newcommand{\norm}[1]{\lVert#1\rVert}
\newcommand{\vzero}{\mathbf{0}}
\newcommand{\ve}{\mathbf{e}}

\newcommand{\mZ}{\mathbf{Z}}

\newcommand{\mB}{\mathbf{B}}
\newcommand{\mZero}{\mathbf{0}}

\newcommand{\mI}{\mathbf{I}}
\newcommand{\mA}{\mathbf{A}}
\newcommand{\mF}{\mathbf{F}}
\newcommand{\mJ}{\mathbf{J}}

\newcommand{\R}{\mathbb{R}}
\newcommand{\cX}{\mathcal{X}}

\newcommand{\vv}{\mathbf{v}}
\newcommand{\cN}{\mathcal{N}}

\newcommand{\cA}{\mathcal{A}}
\newcommand{\cP}{\mathcal{P}}
\newcommand{\cG}{\mathcal{G}}

\newcommand{\cS}{\mathcal{S}}

\newcommand{\cZ}{\mathcal{Z}}
\newcommand{\mK}{\mathbf{K}}

\newcommand{\vh}{\mathbf{h}}
\newcommand{\mS}{\mathbf{S}}

\newcommand{\mR}{\mathbf{R}}
\newcommand{\mL}{\mathbf{L}}

\newcommand{\vOnes}{\mathbf{1}}

\newtheorem{conj}{Conjecture}
\newtheorem{thm}{Theorem}
\newtheorem{lm}{Lemma}
\newtheorem{prop}{Proposition}

\newtheorem{cor}{Corollary}

\definecolor{pc_comment}{RGB}{30,125,125}

\DeclarePairedDelimiter{\paren}{(}{)}
\DeclarePairedDelimiter{\sqbracket}{[}{]}
\DeclarePairedDelimiter{\curlybrace}{\{}{\}}

\usepackage{lipsum}
\usepackage{amsfonts}
\usepackage{graphicx}
\usepackage{epstopdf}
\usepackage{algorithmic}
\ifpdf
  \DeclareGraphicsExtensions{.eps,.pdf,.png,.jpg}
\else
  \DeclareGraphicsExtensions{.eps}
\fi

\newsiamremark{remark}{Remark}
\newsiamremark{hypothesis}{Hypothesis}
\crefname{hypothesis}{Hypothesis}{Hypotheses}
\crefname{thm}{Theorem}{Theorems}
\newsiamthm{claim}{Claim}

\headers{Sigmoidal Opinion Dynamics}{H. Z. Brooks, P. S. Chodrow, and M. A. Porter}

\title{Bifurcations in a Tunably Nonlinear Model of Opinion Dynamics\thanks{Submitted to the editors DATE.
\funding{FUNDING?}}}

\author{Heather Z. Brooks\thanks{Department of Mathematics, Harvey Mudd College, Claremont, CA, USA 
  (\email{hzinnbrooks@g.hmc.edu}).}
\and Philip S. Chodrow\thanks{Department of Computer Science, Middlebury College, Middlebury, VT, USA 
  (\email{pchodrow@middlebury.edu}).}
\and Mason A. Porter\thanks{Department of Mathematics, University of California, Los Angeles, Los Angeles, CA, USA; Santa Fe Institute, Santa Fe, NM, USA 
(\email{mason@math.ucla.edu})}}

\usepackage{amsopn}

\makeatletter
\newcommand*{\addFileDependency}[1]{
  \typeout{(#1)}
  \@addtofilelist{#1}
  \IfFileExists{#1}{}{\typeout{No file #1.}}
}
\makeatother

\title{Emergence of Polarization in a Sigmoidal Bounded-Confidence Model of Opinion Dynamics}

\begin{document}
\maketitle

\begin{abstract}
    We study a nonlinear bounded-confidence model (BCM) of continuous-time opinion dynamics on networks with both 
    persuadable individuals and zealots.
    The model is parameterized by a scalar $\gamma$, which controls the steepness of a smooth influence function. 
    This influence function encodes the relative weights that nodes place on the opinions of other nodes.
    When $\gamma = 0$, this influence function recovers Taylor's averaging model; when $\gamma \rightarrow \infty$, the influence function converges to that of a modified Hegselmann--Krause (HK) BCM.
    Unlike the classical HK model, however, our sigmoidal bounded-confidence model (SBCM) is smooth for any finite $\gamma$.  
     We show that the set of steady states of our SBCM is qualitatively similar to that of the Taylor model when $\gamma$ is small and that the set of steady states approaches a subset
    of the set of steady states of a modified HK model as $\gamma \rightarrow \infty$.
    For several special graph topologies, we give analytical descriptions of important features of the space of steady states. 
    A notable result is a closed-form relationship between the stability of a polarized state and the graph topology in a simple model of echo chambers in social networks. 
    Because the influence function of our BCM is smooth, we are able to study it with linear stability analysis, which is difficult to employ with the usual discontinuous influence functions in BCMs.
\end{abstract}

\section{Introduction} \label{sec:intro}

Collective behavior plays a crucial role in shaping information flow, scientific 
progress, and political decision-making in human societies \citep{Bak-Colemane2025764118}. 
One major thread in the study of human collective behavior is modeling and analyzing
opinion dynamics \cite{noorazar2020classical}.
\emph{Opinion models} encompass simplified social interactions in which agents form and/or refine opinions about a topic through interactions with other agents.
Because humans interact with each other in networked settings, many opinion models situate agents on a graph (or on a more complicated network structure), with direct influence 
between agents who share an edge of the graph. 
These networked opinion models highlight the rich interplay between a system's dynamics and network structure \cite{porter2016}.
{For discussions of models of opinion dynamics from different perspectives, see \citet{bullo2019lectures}, \citet{golubitsky2023}, \citet{noorazar2020classical}, and \citet{proskurnikov2017tutorial}.}

One prominent model of opinion dynamics is the French--DeGroot (FD) model. 
The FD model has its roots in the work of \citet{french1956formal} on social power, and it was later generalized by \citet{degroot1974reaching} to the study of consensus in a collection of rational agents. In the FD model, each agent (i.e., node) $i$ has a scalar opinion $x_i(t)\in \R$ at discrete time $t$. We encode the set of node opinions as
an \emph{opinion vector} (i.e., \emph{opinion state}) $\vx(t) \in \R^n$, where $n$ is the number of agents. (In this paper, we use the terms ``opinion vector" and ``opinion state" interchangeably.)
The action of a row-stochastic matrix $\mB \in \R^{n\times n}$ yields synchronous opinion updates in discrete time steps:
\begin{align*}
    \vx(t+1) = \mB\vx(t)\;. 
\end{align*}
Typically, the matrix $\mB$ is a (possibly weighted) adjacency matrix, with an associated graph
$\cG$ (which we assume is undirected), such that $b_{ij} > 0$ 
only if $(i,j)$ is an edge of $\cG$. (The FD model permits edges of weight $0$.)
The graph $\cG$ encodes a social network, so we call it a \emph{social graph}.
The time-($t+1$) opinion $x_i(t+1)$ of node $i$ is a weighted average of 
the opinions of node $i$ and its neighbors at time $t$.

The Abelson model \citep{abelson1964mathematical, abelson1967mathematical} is a continuous-time variant of the FD model. It is given by the dynamical system
\begin{align*}
    \dot{\vx}(t) = \tilde{\mB}\vx(t)\;,
\end{align*}
where $\tilde{\mB}$ is a
matrix that depends on the underlying social graph $\cG$.
For example, with the choice $\tilde{\mB} = \mB - \mI$, the dynamics of the Abelson model are qualitatively similar to those of the corresponding FD model. 
In particular, the steady states of these two models are identical. 
When $\mB$ or $\tilde{\mB}$ are irreducible, as often is the case for matrices that one obtains from a connected graph $\cG$, these models converge to a \emph{consensus} steady state, which satisfies the property that $x_i = x_j$ for all agents $i$ and $j$. 
{In studies of opinion models, it is common to investigate
whether or not their dynamics converge to a steady state,} 
the speed of such convergence,
  and 
whether or not the set of steady states includes a consensus state~\citep{noorazar2020classical}.

Although the FD and Abelson models often converge
to consensus, it is rare to observe consensus in many real-world social systems \citep{levinDynamicsPoliticalPolarization2021}. 
This motivates the study of opinion models that exhibit enduring \emph{dissensus}, which encompasses polarization (i.e., two major opinion groups at steady state), fragmentation (i.e., three or more major opinion groups at steady state), and other situations. The emergence of dissensus is not an inherently negative outcome, as the social utility of an opinion state depends on its context \citep{landemore2015deliberation}. 
One approach to modeling persistent dissensus is to introduce agents whose opinions do not change with time; such agents are often called \emph{zealots} or ``stubborn agents". 
The Friedkin--Johnsen (FJ) model generalizes the FD model to include zealots \citep{friedkin1990social}, and the Taylor model \citep{taylor1968towards} analogously generalizes the Abelson model. 
In both extensions, the presence of at least two zealots with different opinions is sufficient to prevent global consensus at steady state.
The introduction of zealots leads to rich behavior in a variety of opinion models, including naming-game models \citep{verma2014impact}, voter models \citep{mobilia2007role, klamser2017zealotry}, and Galam models \citep{galam2007role, martins2013building}.  

Another way to obtain dissensus in opinion dynamics is by 
incorporating complexities into the rules that govern interactions between agents.
Hegselmann and Krause  \citep{hegselmann2002opinion, hegselmann2015opinion, hegselmann2019consensus} incorporated a nonlinearity into the FD
averaging model by introducing a ``confidence bound'' $\delta > 0$. 
In the Hegselmann--Krause (HK) model, which is a type of ``bounded-confidence model'' (BCM), the matrix $\mB$ is no longer fixed; it now depends on the opinion state $\vx$. 
In particular, $b_{ij} > 0$ 
only when agents $i$ and $j$ are adjacent (i.e., connected directly to each other in a network) and have opinions that satisfy $\abs{x_i - x_j} < \delta$. 
A convenient way to express this idea is by defining an \emph{influence function} $w:\R^2 \rightarrow \R$ 
with the formula
\begin{align}
    w(x_i, x_j) = \mathbbm{1}\left[\abs{x_i - x_j} < \delta\right]\;, \label{eq:HK-influence}
\end{align}
where $\mathbbm{1}$ is the indicator function. The influence function
$w(x_i, x_j) = 1$ if $\abs{x_i - x_j} < \delta$ and $w(x_i, x_j) = 0$ otherwise.
With this choice, we may write $b_{ij} = \hat{b}_{ij}w(x_i, x_j)$, where $\hat{b}_{ij}$ is the weight of the edge between nodes $i$ and $j$ when $\abs{x_i - x_j} < \delta$.
If we instead choose the constant influence function $w(x_i, x_j) = 1$, we obtain the Abelson model.

In the HK model, neighboring nodes with sufficiently different opinions do not interact with each other (or, at least, their opinions do not move closer to each other
as a result of an interaction). 
The HK model converges to a limiting opinion vector (i.e., a steady state)
in a finite number of time steps \citep{dittmer2001consensus}. The Deffuant--Weisbuch (DW) BCM has a similar rule for opinion updates, but only two adjacent agents update their opinions in each time step \citep{deffuant2000mixing}. 
Under suitable conditions, the DW model 
converges in expectation exponentially quickly in time to a steady state~\cite{chen2020convergence}. 
In both the HK and DW models, the structure of the steady state that occurs in
a single simulation depends nontrivially on the confidence bound $\delta$, the topology of the underlying social graph $\cG$, and the initial node opinions \citep{hegselmann2002opinion,meng2018opinion}. 
Because of the discontinuous dependence of interactions on the distances between node opinions, 
traditional\footnote{Linear stability analysis has been extended to piecewise-smooth dynamical systems \citep{bernardo2008piecewise}.}
linear stability analysis is typically
unhelpful. See \citet{ceragioli2021generalized} for {analyses of the asymptotic behavior of a wide class of BCMs and rigorous characterizations of the conditions for existence and uniqueness of their steady states.}

Some researchers have examined more complicated notions of stability, such as stability with respect to the introduction of new agents~\citep{blondel2009krause}, in BCMs. 
It is also possible to formulate variants of BCMs that incorporate zealots, with consequences for the structure of their steady states. 
For example, \Citet{fu2015opinion}
 observed that introducing zealots into the HK model affects both the number of steady-state opinion groups
   and the size of the largest steady-state opinion group.
 \Citet{brooks2020model} demonstrated numerically that the number of agents that agree with zealots 
  at steady state depends nonmonotonically on the number of zealots.

In this paper, we present a parameterized, nonlinear, continuous-time model of opinion dynamics that interpolates smoothly between the Abelson model and the HK
model. Our model is a \emph{sigmoidal bounded-confidence} model (SBCM). 
Using our SBCM, we {interpolate between}
an averaging model and a BCM as we vary its parameters. 
Our work complements several existing ``smoothed'' and otherwise continuous variants of the HK model.{\footnote{Researchers have also studied a variety of continuous variants of the DW model~\cite{deffuant2002,deffuantsmooth2004,gandica2023}.}
\Citet{yang2014opinion} considered a smooth influence function (also see \cite{olfati2003}) as a numerical approximation of the HK model, and they examined convergence properties of the resulting model.
\Citet{ceragioli2012continuous} carefully studied the role of the discontinuity in the HK model by considering sequences of ``smooth'' HK systems. 
Using this strategy, they proved that solutions of these smooth systems exist for all initial conditions, detailed qualitative similarities between the smoothed systems and the traditional HK model,
and showed that sequences of solutions of the smoothed systems converge pointwise to solutions of the classical HK model.  
However, they did not consider the role of zealots, nor did they study the effect of changing the influence function on the qualitative structure of the space of steady states.  

Several recent papers have analyzed networked opinion models
with a variety of nonlinear influence functions.
\citet{franci2019model} and \citet{bizyaeva2020nonlinear} formulated a flexible family of models and used tools from equivariant bifurcation theory to study consensus and dissensus behaviors that arise from particular system symmetries. \citet{bonetto2022nonlinear} recently investigated the effects of symmetries on consensus for systems with nonlinear Laplacian dynamics. \citet{devriendt2021nonlinear} studied the steady-state behavior of a gradient dynamical system whose dynamics are governed by an odd coupling function of the distance between node opinions.
In \cite{homs2020nonlinear}, these authors and Homs-Dones studied the steady states 
of this system using effective-resistance techniques. They derived closed-form expressions for steady
states on networks that are trees, cycles, or complete graphs.
\citet{neuhauser2020multibody} examined the effects of nonlinear interaction functions on consensus dynamics 
on networks with three-node interactions.
The model that we study is
related to the nonlinear gradient systems that have arisen in other applications, including synchronization of coupled oscillators \cite{arenas2008synchronization,rodrigues2016kuramoto}, collective behavior in animals \cite{couzin2011uninformed, nabet2009dynamics}, and aggregation dynamics \cite{bertozzi2009blow}.

Our article proceed as follows. 
In \Cref{sec:modeldef}, we define our SBCM on networks with zealots. 
Our SBCM's influence function includes a tunable parameter $\gamma \in \mathbb{R}_{\geq 0}$ with extreme values $\gamma = 0$ and $\gamma \rightarrow \infty$ that correspond to the influence functions of the Abelson model and a modified HK model, respectively. 
In \Cref{sec:basicresults}, we discuss the structure of the linearization of our SBCM. 
In \Cref{sec:limitingprops}, we study our SBCM's steady states in the limiting cases $\gamma = 0$ and $\gamma \rightarrow \infty$. We are concerned especially with the latter case, and we describe circumstances in which the steady states of our SBCM in this regime
resemble steady states of the HK model. 
In \Cref{sec:specialcases}, we examine the qualitative behavior of the space of steady states as one varies $\gamma$. 
We do not give general results in this situation, but we are able to make progress for some special graph structures.
We consider two special graph structures and give analytical descriptions of some properties of our SBCM's steady states 
for these cases. 
A notable result is an analytical description of the relationship between network topology and the linear stability of polarized opinion states in a simple scenario that is motivated by echo chambers in social networks. We conclude in \Cref{sec:conclusions} with a discussion and suggestions for future work. In \Cref{software}, we indicate what software we employed and give a website with a code repository to reproduce our numerical experiments.

\section{Our sigmoidal bounded-confidence model (SBCM)} \label{sec:modeldef}

    Let $\mA \in \{0,1\}^{n \times n}$ be the adjacency matrix of an undirected, unweighted graph $\cG$ with a set $\cN$ of $n$ nodes. 
    Nodes $i$ and $j$ are adjacent if $a_{ij} = 1$; when $i$ and $j$ are not adjacent, $a_{ij} = 0$.
    We use the notation $i \sim j$ to indicate that nodes $i$ and $j$ are adjacent. 
    We assume that $\cG$ has no self-edges, so
        $a_{ii} = 0$ for all $i \in \cN$.
    
    An \emph{opinion state} (or, equivalently, an \emph{opinion vector}) of our SBCM is a vector $\vx \in \R^{n}$. 
    The entry $x_i$ is the \emph{opinion} of node $i$. 
    We write $\vx = \vx(t)$ when we wish to emphasize the dependence of $\vx$ on time $t$. 
    We assume that some subset $\cZ \subset \cN$ of nodes are \emph{zealots}; the opinions of these nodes are not influenced by other nodes.
        By contrast, the opinions of \emph{persuadable} nodes $\cP = \cN \setminus \cZ$ 
        can change when they interact with other nodes.
    
    We define opinion dynamics on $\mA$ via an update operator $\mF$ with components $\{f_i(\vx)\}$ that govern the time evolution of the system: 
    \begin{align}
        \frac{d x_i}{d t} = f_i(\vx) \triangleq 
        \begin{cases}
            \frac{\sum_j w(x_i, x_j)(x_j - x_i)}{\sum_j w(x_i, x_j)}\,, &\quad i \in \cP \\
            0\,, &\quad i \in \cZ\;, \label{eq:dynamics}
        \end{cases}
    \end{align}
    where the \emph{influence function}
    $w : \R^2 \rightarrow \R$ encodes
        the susceptibility of nodes to each other's opinions as a function of their current opinions. 
    
    In our SBCM, we consider the parameterized influence function
    \begin{align}
    \label{eq:w_sigmoid}
        w(x_i, x_j) \triangleq 
        \begin{cases}
            \frac{1}{1 + e^{\gamma(x_i - x_j)^2 - \gamma \delta }}\,, &\quad i \sim j \\ 
            0\,, &\quad \text{otherwise}\,,
        \end{cases}
    \end{align} 
    where $\gamma, \delta \in \R_{\geq 0}$.
            This function is a translated and reflected logistic sigmoid with respect to the square of the distance between node opinions.\footnote{This sigmoidally smoothed influence
            function is similar to the one that was used by \citet{okawa2022predicting} [see equation (8) of their paper] in their construction of a neural network that is informed by models of opinion dynamics. However, Okawa and Iwata did not further examine the properties of their \acc{model}.} 
          Heuristically, for adjacent nodes $i$ and $j$, the influence
          function $w(x_i, x_j)$ is large when $x_i$ and $x_j$ are close in opinion space (i.e., when nodes $i$ and $j$ have similar opinions). 
    One can interpret $w(x_i, x_j)$ as a weighting function that encodes the relative receptivity of nodes $i$ and $j$ to
   each other's opinions.\footnote{The evolution of weights as a function of opinions as a feedback mechanism is
    reminiscent of the evolution of self-weights in the DeGroot--Friedkin model \cite{jia2015opinion}.}  
    We collect these values in a matrix $\mW(\vx) \in \R^{n \times n}$. 
    When an opinion state $\vx$ is clearly implied, we abbreviate the components of this matrix as $w_{ij} = w(x_i, x_j)$ and we abbreviate the matrix itself as $\mW = \mW(\vx)$. 
    Additionally, $w(x_i,x_j)$ depends on $x_i$ and $x_j$ only through their absolute difference $\abs{x_i - x_j}$. 
    Therefore, we can define a function $\omega: \R \rightarrow \R$ 
    through the relation $w(x_i, x_j) = \omega(\abs{x_i - x_j})$.

    \begin{figure}
        \centering
        \includegraphics[width=\textwidth]{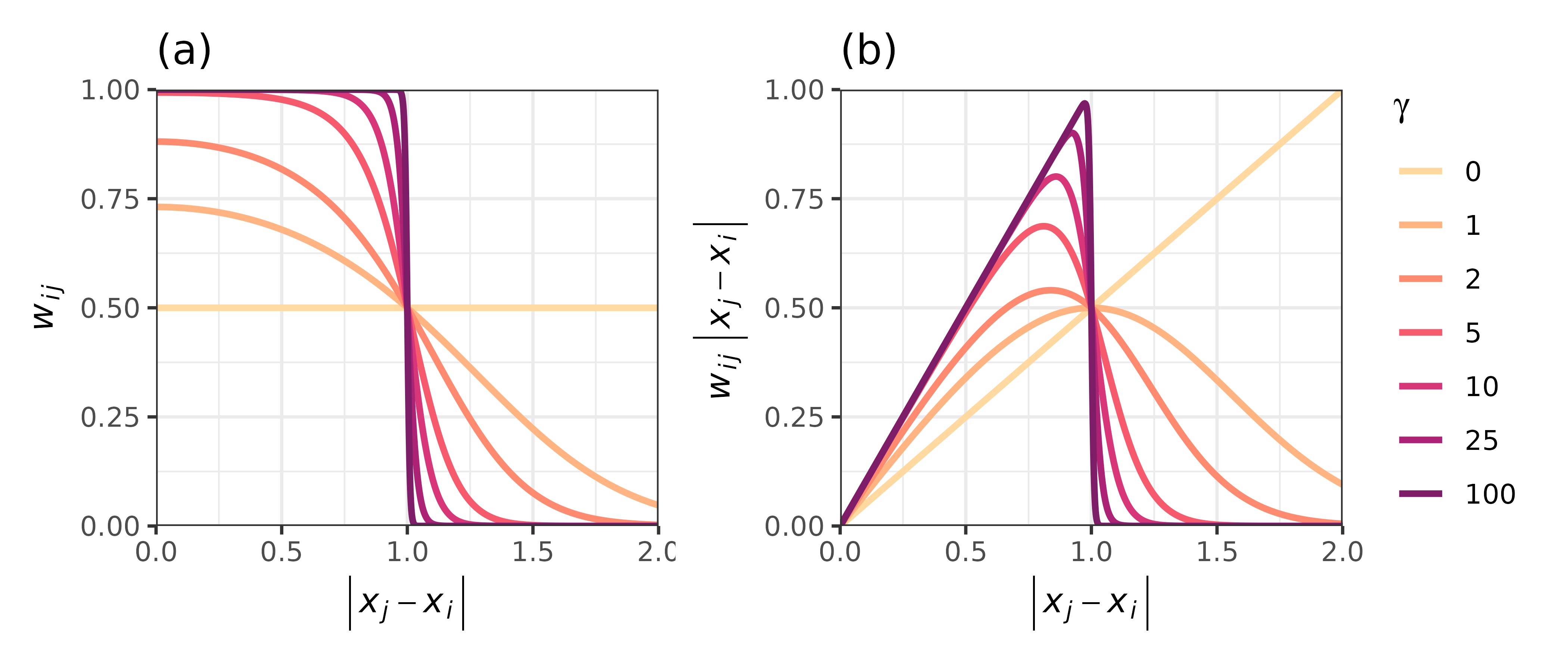}
        \caption{
            (a) Examples of the 
                        influence function $w(x_i, x_j)$ with the confidence bound $\delta = 1$ for different values of the parameter $\gamma$. 
            When $\gamma=0$, all interaction weights are equal and independent of the distances between opinions. 
                        As $\gamma$ increases, there is an increase in the interaction strength between nodes
                        $i$ and $j$ whose opinions satisfy $\vert x_j - x_i \vert^2 < \delta$. 
            Conversely, as $\gamma$ increases, the interaction strength decreases between nodes
            whose opinions satisfy $\vert x_j - x_i \vert^2 > \delta$.
                                (b) Examples of the influence 
                                on the opinion of a persuadable node as one increases $\gamma$. This influence of node $j$ on node $i$ depends on the product of the 
                        influence function $w(x_i,x_j)$ in the left panel and the distance $\vert x_j - x_i \vert$ in opinion space. 
            When $\gamma = 0$, the influence is monotonic with respect to the the opinion distance $\vert x_j - x_i \vert$; when $\gamma > 0$, the influence has a local maximum. 
        } \label{fig:sigmoid}
    \end{figure}

    The parameter $\gamma$ controls the ``sharpness'' of the dependence of $w_{ij}$ on the squared distance $(x_i - x_j)^2$ (see \Cref{fig:sigmoid}). 
    Two limiting cases are of particular interest. 
    When $\gamma = 0$, the function $w(x_i, x_j)$ is a constant and is thus independent of $x_i$ and $x_j$. 
    By contrast, as $\gamma \rightarrow \infty$, the function $w_{ij}$ converges pointwise to a step function (with the step located at $\delta$) of the squared distance: 
    \begin{align}
    \label{eq:winfty}
        w^{(\infty)}(x_i, x_j) = \lim_{\gamma \rightarrow \infty} w(x_i, x_j) = 
        \begin{cases}
            1\,, &\quad (x_i - x_j)^2 < \delta \\ 
            \frac{1}{2}\,, &\quad (x_i - x_j)^2 = \delta \\ 
            0\,, &\quad (x_i - x_j)^2 > \delta\,. 
        \end{cases}
    \end{align} 
    These two limits correspond to well-known opinion models. 

    When $\gamma = 0$, equation \eqref{eq:dynamics} reduces to
    \begin{equation}
    \label{eq:Taylor}
        \frac{d x_i}{d t} = \frac{1}{d_i}\sum_{j\sim i}(x_j - x_i)\;,
    \end{equation}
    where $d_i$ is the degree (i.e., the number of neighbors) of node $i$.
    This is a simple version of the Abelson model \cite{abelson1964mathematical, abelson1967mathematical}.  
    In the absence of zealots, the only steady
        states of \eqref{eq:Taylor} on connected graphs are consensus states, in which all nodes hold the same opinion at stationarity and $\vx = x\vOnes$ \cite{ren2005consensus}.\footnote{
        {One can write the Abelson model as $\dot{\vx} = \mD^{-1}\mL\vx$, where $\mD$ is the diagonal matrix of node degrees and $\mL = \mD - \mA$ is the combinatorial graph Laplacian. 
        Therefore, our model \eqref{eq:dynamics} extends one type of linear Laplacian dynamics. 
        For several extensions of models of the form $\dot{\vx} = \mL\vx$, see the generalized Laplacian-flow models of \citet{bonetto2022nonlinear} and \citet{srivastava2011bifurcations}.}}
    In several of our arguments, it is
          useful to explicitly track the dependence of the update operator $\mF$ on $\gamma$. We do this by
          using the notation $\mF_\gamma$.

    As long as $\cZ$ is not empty (i.e., when the graph $\cG$ has at least one zealot), equation \eqref{eq:Taylor} is an extension of the Abelson model {that has been attributed to Taylor \cite{taylor1968towards}}. If 
    $\cG$ is connected, equation \eqref{eq:Taylor} has
        a unique steady
        state \cite{proskurnikov2017tutorial}.
         This steady state is \emph{harmonic} at all persuadable vertices in the sense that the opinion of each persuadable node is equal to the mean of the
    opinions of its neighbors \cite{karlsson2003some}. 
    We refer to this steady state, which we denote by $\bar{\vx}$,
         as the \emph{harmonic 
    state} because it is the discrete harmonic extension of zealot opinions to the rest of $\cG$
        \cite{benjamini2003harmonic}. 
    The harmonic state depends only on the zealot opinions \cite{proskurnikov2017tutorial}, and equation \cref{eq:Taylor}
    reaches this state regardless of the initial opinions 
        of the persuadable nodes. 
        The harmonic state coincides with the steady state of FD
        dynamics on a connected graph with zealots~\cite{parsegov2016novel, proskurnikov2017tutorial}. 
    
    When $\gamma \rightarrow \infty$, the update operator
    $\mF$ converges pointwise to the update operator
    of a continuous-time BCM with synchronous updating. This limiting model, with an update
    operator that we denote by $\mF_\infty$, is in the spirit of the well-known HK
        model \cite{hegselmann2002opinion, hegselmann2015opinion, hegselmann2019consensus}. 
    Notably, there is a qualitative difference between the steady states of our
    SBCM as $\gamma \rightarrow \infty$ and those of the classical HK model. 
    The limiting influence function \eqref{eq:winfty} differs from the influence function in the HK model when $(x_i-x_j)^2 = \delta$. 
    The HK model already possesses multiple steady states,
         and this modification introduces additional ones. 
    As we will show in \Cref{sec:limitingprops}, as $\gamma \rightarrow \infty$, the set of linearly stable steady states of our SBCM converges to 
        a subset 
    of the steady states of the standard HK model. 
        If a sequence of steady states of our SBCM converges to one of
        the \emph{additional} steady states of the modified HK model that we obtain by using
        the modified influence function \eqref{eq:winfty}, then these
        steady states are
        linearly unstable for sufficiently large $\gamma$.

    \begin{figure}
        \centering
        \includegraphics[width=\textwidth]{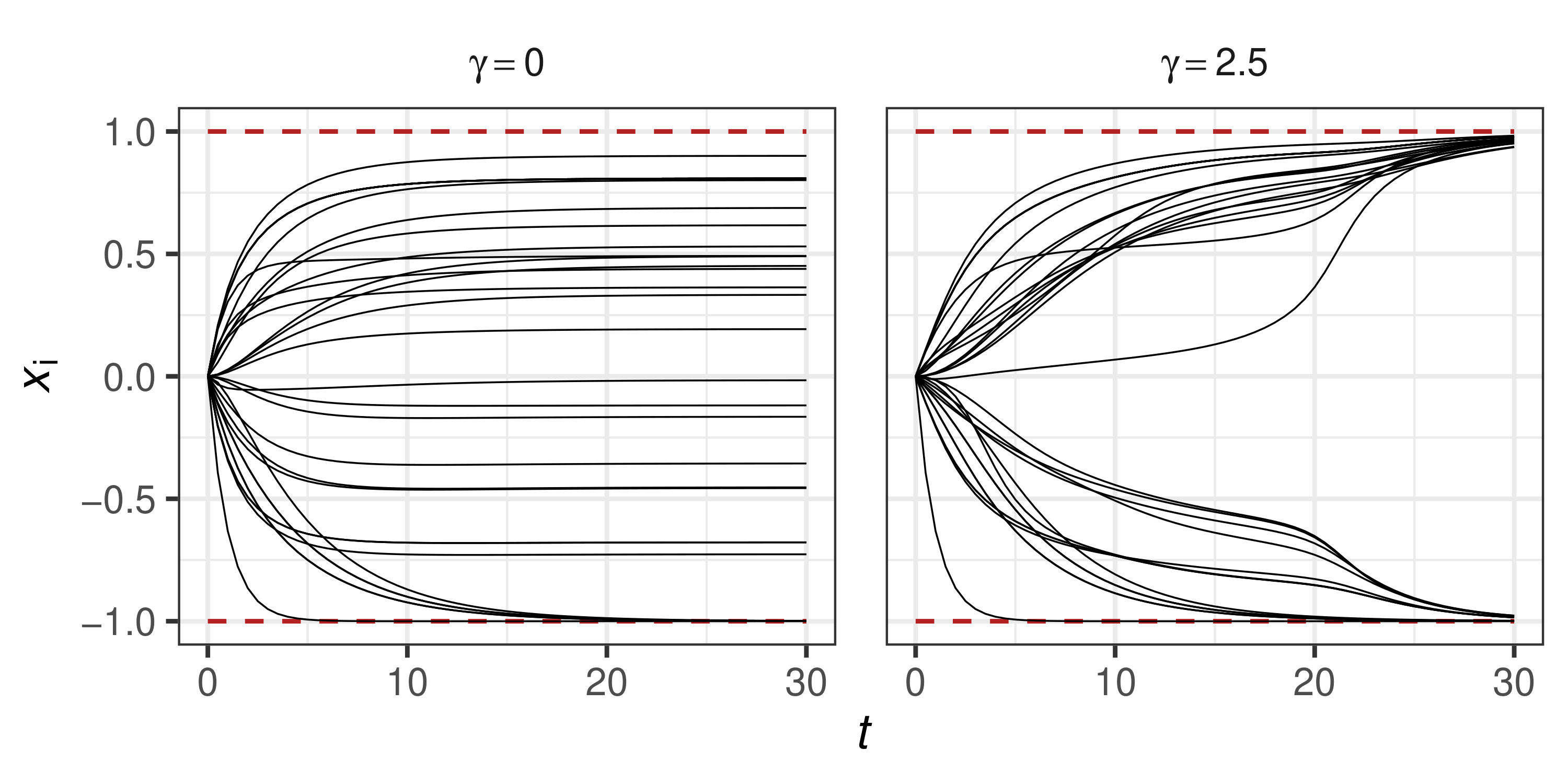}
        \caption{Time evolution of node opinions for two different values of $\gamma$ in the Zachary Karate Club graph \cite{zachary1977information}.
        We initialize the opinion of each persuadable node $i$ to be $x_i = 0$. 
        When $\gamma = 0$, there is a unique steady state in which the nodes are scattered widely
        in opinion space.      
        As $\gamma$ increases, sharply polarized steady states emerge.  
        The solid curves show the opinion trajectories of persuadable nodes, and the dashed lines show the opinion trajectories of zealots.}
        \label{fig:dynamics-examples}
    \end{figure}

    In \Cref{fig:dynamics-examples}, we give two examples of the time evolution of our SBCM on a small graph. 
    When $\gamma = 0$, the unique steady state is the harmonic state. 
 In the harmonic state, the node opinions are
        scattered 
        widely 
        between the two zealots. 
        As we increase $\gamma$, we also obtain other steady states, including sharply polarized ones.

\section{Steady states and linear stability analysis: Basic results} \label{sec:basicresults}

    We now examine the steady states of 
        equation \eqref{eq:dynamics}, which (for convenience) we also describe as ``steady states of $\mF$".
    We present basic results about these states and their linear stability.

    We explore the structure and stability of steady states of our SBCM by examining the linearization of equation \eqref{eq:dynamics} via the Jacobian matrix $\mJ$ of $\mF$. We evaluate $\mJ$ at a steady state $\vx$. We begin by separating $\mJ$
        into blocks with different combinations of persuadable nodes $\cP$ and zealots $\cZ$. 
    We write
    \begin{align}
        \mJ = \sqbracket*{\begin{matrix}
            \mJ_{\cP} & \mJ_{\cP\cZ} \\   
            \mJ_{\cZ\cP} & \mJ_{\cZ}
        \end{matrix}} = 
         \sqbracket*{\begin{matrix}
            \mJ_{\cP} &             \mJ_{\cP\cZ}\\   
            \mZero & \mZero 
        \end{matrix}}\;.
    \end{align}
    The entries of the block $\mJ_\cP$ are derivatives of the form $\partial f_i(\vx)/\partial x_j$ with $i, j \in \cP$, the entries of the block $\mJ_{\cP\cZ}$ are derivatives of the form $\partial f_i(\vx)/\partial x_j$ (with $i \in \cZ$ and $j \in \cP$, and the entries of the other blocks are analogous.
    The two lower blocks are identically $0$ because the dynamics of equation \eqref{eq:dynamics}
    does not change the opinions of the zealots. 
    The spectrum of $\mJ$ thus consists of the spectrum of $\mJ_{\cP}$ and a set of $0$ eigenvalues that correspond to prohibited modifications of zealot opinions. 
    
    Given the above structure of $\mJ$, we focus our attention on $\mJ_{\cP}$, which we evaluate at a steady state $\vx$.  
      Fix $i,j \in \cP$. 
    If $i \not \sim j$ and $i \neq j$, we know that $\frac{\partial f_i(\vx)}{\partial x_j} = 0$. 
    Therefore, assume that $i \sim j$ and $i \neq j$. For convenience, we define the strength (i.e., weighted degree) of node $i$ to be $s_i \triangleq \sum_j w_{ij}$.
        
    Entry $(i,j)$ of $\mJ_\cP$ is
    \begin{align*}
        \frac{\partial f_i(\vx)}{\partial x_j} = \sqbracket*{\frac{\partial}{\partial x_j}\paren*{\frac{1}{s_i}}}\underbrace{\sum_{k\sim i}{w_{ik}(x_k - x_i)}}_{= s_i\mF(\vx)_i} + \frac{1}{s_i} \sum_{k \sim i} \paren*{\frac{\partial w_{ik}}{\partial x_j}(x_k - x_i) +  w_{ik} \frac{\partial x_k}{\partial x_j} }\,.
    \end{align*}  
    Because $\vx$ is a steady state of equation \eqref{eq:dynamics} by hypothesis, the first term vanishes. Therefore,
    \begin{align}
    \label{eq:J_offdiag}
        \frac{\partial f_i(\vx)}{\partial x_j} = \frac{1}{s_i}  \paren*{\frac{\partial w_{ij}}{\partial x_j}(x_j - x_i) +  w_{ij}}\,.    
    \end{align}
    This gives an expression for the off-diagonal entries of $\mJ_\cP$.

    The diagonal entries are
    \begin{align}
    \label{eq:J_diag}
        \frac{\partial f_i(\vx)}{\partial x_i} &= \frac{1}{s_i}\sum_{k \sim i}\paren*{\frac{\partial w_{ik}}{\partial x_i}(x_k - x_i) - w_{ik}} \\ 
        &= \frac{1}{s_i} \sqbracket*{ \sum_{\substack{k \sim i \\ k \notin \mathcal{Z}}}\paren*{\frac{\partial w_{ik}}{\partial x_i}(x_k - x_i) - w_{ik}} +  \sum_{\substack{k \sim i \\ k \in \mathcal{Z}}}\paren*{\frac{\partial w_{ik}}{\partial x_i}(x_k - x_i) - w_{ik}}} \,. \nonumber
    \end{align}
    
    Thus far, we have refrained from using the functional form of $w$ in
        equation \eqref{eq:w_sigmoid}. Indeed, in future work, it may be desirable to consider choices of $w$ other than ours.
                        Incorporating assumptions about the structure of $w$ allows us to make further simplifications.
    If we suppose that  $w$ is a function of $x_i$ and $x_j$ only through $x_i - x_j$ (as is the case in equation \eqref{eq:w_sigmoid}), 
    it follows that $\frac{\partial w_{ik}}{\partial x_i} = -\frac{\partial w_{ik}}{\partial x_k}$. 
            We then 
        express the diagonal terms via the off-diagonal terms by writing
    \begin{align*}
        \frac{\partial f_i(\vx)}{\partial x_i} = - \sum_{\substack{j\sim i \\ j \notin \mathcal{Z}}}\frac{\partial f_i(\vx)}{\partial x_j} - \sum_{\substack{j\sim i \\ j \in \mathcal{Z}}}\frac{\partial f_i(\vx)}{\partial x_j}\,.
    \end{align*}

    To make further progress, we explicitly calculate the derivative 
    \begin{align*}
        \frac{\partial w_{ij}}{\partial x_j} = -2\gamma w_{ij}(1 - w_{ij})(x_j - x_i)
    \end{align*}
    to obtain 
    \begin{align*}
        \frac{\partial f_i(\vx)}{\partial x_j} = \frac{w_{ij}}{s_i}\sqbracket*{1 - 2\gamma (1-w_{ij})(x_j - x_i)^2 }\,.
    \end{align*}
    We now define several matrices, which we implicitly evaluate at the steady state $\vx$. 
    Let $\mR_{\cP}(\gamma)$ be the matrix with entries $r_{ij} = s_i \frac{\partial f_i}{\partial x_j}$, and 
    let $\mL_\cP(\gamma)$ be the combinatorial Laplacian matrix of $\mR_{\cP}$. Entry $(i,j)$ of $\mL_\cP(\gamma)$
    is $\mathbbm{1}[i = j]\sum_{k \sim i} r_{ik} - r_{ij}$, where the indicator function $\mathbbm{1}[i = j]$ is 1 if $i = j$ and is 0 otherwise. 
    We collect the strengths $s_i$ into a diagonal matrix $\mS_{\cP}$, and we let $\mZ_{\cP}$ be the diagonal matrix with entries $z_{ii} = s_i\sum_{j \in  \cZ(i)} \frac{\partial f_i}{\partial x_j}$. 
    We then write the Jacobian matrix for the subgraph of persuadable nodes (i.e., the so-called ``persuadable subgraph'') as 
    \begin{align} \label{eq:matrix-jacobian}
        \mJ_\cP \triangleq \frac{\partial \mF}{\partial \vx} = - \mS_{\cP}^{-1}\sqbracket*{\mZ_\cP + \mL_\cP}\;. 
    \end{align}
    Because $\mS_{\cP}^{-1}$ is symmetric and positive definite, its square root is well-defined and $\mJ_\cP$ is similar to the symmetric matrix $\mS_{\cP}^{-1/2} \mM_\cP \mS_{\cP}^{-1/2}$, where  $\mM_\cP\triangleq -\mZ_\cP - \mL_\cP$. 
    From this, we infer two important facts. 
    First, the eigenvalues of $\mJ_\cP$ are real. 
    Second, $\mJ_\cP$ is negative definite (respectively, negative semidefinite) if and only if $\mM_\cP$ is negative definite (respectively, negative semidefinite).

    Because the entries of $\mR_{\cP}$ are not guaranteed to be nonnegative, $\mL_\cP$ is not necessarily positive definite. 
    However, one can write $\mL_\cP$ as the difference between two positive-semidefinite matrices:
        \begin{align*}
        \mL_{\cP} = \mL_{\cP}^{(1)} - 2\gamma \mL_{\cP}^{(2)}\;,
    \end{align*}
    where $\mL_{\cP}^{(1)}$ is the combinatorial Laplacian of the matrix with entries $w_{ij}$ and $\mL_{\cP}^{(2)}$ is the combinatorial Laplacian of the matrix with entries $w_{ij}(1-w_{ij})(x_j - x_i)^2$. 
   
    The following proposition gives a sufficient condition for the existence of a positive eigenvalue of $\mJ$.

    \begin{prop} \label{prop:instability-condition}
        At a steady state $\vx$, suppose that there exists a node $i$ such that 
        \begin{align} 
            (1 - w_{ij})(x_j - x_i)^2 > \frac{1}{2\gamma} \quad \text{for all} \quad j \sim i\;. \label{eq:instability-condition}
        \end{align}
        It then follows that $\mJ_\cP$ (and thus the Jacobian matrix $\mJ$) at $x$ has at least one strictly positive eigenvalue. 
    \end{prop}

    \begin{proof}
        Because $\mJ_\cP$ and $\mM_\cP$ are similar matrices, it suffices to show that $\mM_\cP$ has a strictly positive eigenvalue.  
        Let $\lambda$ be the largest eigenvalue of $\mM_\cP$. 
        Because $\mM_\cP$ is symmetric, the Rayleigh--Ritz Theorem gives the lower bound $\lambda \geq \vv^T\mM_\cP\vv$ for any unit vector $\vv$. 
        Choose $\vv = \ve_i$. 
        Because $\ve_i^T\mM_\cP\ve_i = m_{ii}$, it suffices to check whether or not $\mM_\cP$ has any positive diagonal entries. 
        The diagonal entries $m_{ii}$ are
        \begin{align*}
            m_{ii} &= - \sum_{j \sim i}w_{ij}\sqbracket*{1 - 2\gamma(1-w_{ij})(x_j - x_i)^2}\;.
        \end{align*}
        If this sum is strictly positive, then we have proven the existence of a positive eigenvalue of $\mJ$. 
        A heavy-handed sufficient condition for this is that each individual term of the sum is positive. 
                This condition is exactly 
        \begin{align*}
            (1-w_{ij})(x_j - x_i)^2 > \frac{1}{2\gamma} \quad \text{for all} \quad j \sim i\,.
        \end{align*}
                        This proves the claim. 
    \end{proof}

    Recall that a steady state $\vx$ is linearly unstable if $\mJ$ at the steady state $\vx$ has at least one positive eigenvalue. 
    If $\vx$ is a linearly stable steady state, it then follows that each node $i$ has a neighbor $j$ such that $(1-w_{ij})(x_j-x_i)^2 < \frac{1}{2\gamma}$.

\section{Limiting behavior} \label{sec:limitingprops}

    In \Cref{sec:modeldef}, we discussed the relationship between the update 
     operator of our SBCM and those of other well-studied opinion models. 
         When $\gamma = 0$, our
    SBCM reduces to the Taylor model, which encodes continuous-time averaging of the opinions of the nodes of a network. 
    In the limit $\gamma \rightarrow \infty$, the SBCM update operator
    converges pointwise to that of an HK model.
        One can
        imagine, for sufficiently small $\gamma$, that 
        our SBCM's long-term dynamics resemble those of a continuous-time averaging model and, for sufficiently large $\gamma$, that its
        long-term dynamics resemble those of a continuous-time HK model. 
    In this section, we give several precise statements that support this intuition.

    \subsection{Small $\gamma$}

    As we discussed in \Cref{sec:modeldef}, when $\gamma = 0$, our SBCM is identical to the Taylor continuous-time averaging model \cite{taylor1968towards}. 
        In particular, there is a unique steady state, which is the harmonic state $\bar{\vx}$. The Implicit Function Theorem implies that
        small perturbations of $\gamma$ 
        from $0$ result in a qualitatively similar system in the sense that the perturbed system possesses a unique steady state that depends continuously on $\gamma$.

    \begin{thm}\label{thm:small-gamma}
        Let $\cG$ be a graph with persuadable nodes $1, \ldots, n$ and zealots $n + 1, \ldots, n + m$. 
        Suppose that $\cG$ has at least one zealot and that the persuadable subgraph $\cG_\cP$ is connected.
        Let $\bar{\vx}$ be the unique solution of 
        our SBCM with $\gamma = 0$ on \acc{$\cG$}.
        It follows that there exists $\epsilon > 0$ and a unique $C^{1}$ function $\vh: [0,\epsilon) \rightarrow \R^{n}$ such that the following statements hold:
        \begin{enumerate}
            \item The function $\vh$ satisfies $\vh(0) = \bar{\vx}$\,.
            \item For all $\gamma \in [0,\epsilon)$, equation \eqref{eq:dynamics} has
            a unique $\gamma$-dependent steady state $\vx_\gamma \in \R^{n}$ and the relation $\vh(\gamma) = \vx_\gamma$ holds. 
                 \item For all $\gamma \in [0, \epsilon)$, the steady state $\vx_\gamma$ is linearly stable. 
                    \end{enumerate}
    \end{thm}

    \begin{proof}
        We apply the Implicit Function Theorem \cite{munkres2018analysis} to the function $\mH: \R^{n+1} \rightarrow \R^n$ that is given by $\mH(\gamma, \mathbf{x}) = \mF(\vx)$, where we explicitly treat the parameter $\gamma$ in equation \eqref{eq:dynamics} as an argument of $\mathbf{H}$.  
        It suffices to show that the Jacobian matrix of $\mH$ has full rank when $\gamma = 0$. 
        To demonstrate this, it suffices to verify that $\mM_\cP$ (and thus $\mJ_\cP$) is negative definite.

        When $\gamma = 0$, the nonzero entries of the matrix $\mR_{\cP}$ are $r_{ij} = {1}/{2}$, so $2\mL_{\cP}$ is the combinatorial Laplacian matrix of the adjacency matrix $\mA_\cP$ of the persuadable subgraph.
                The graph $G_\cP$ is connected, by hypothesis, so $\mL_{\cP}$ is positive semidefinite with a unique $0$ eigenvalue that corresponds to the eigenvector $\vOnes$. 
        However, $\mZ_{\cP}$ is also positive semidefinite. Additionally, $\vOnes^T\mZ_{\cP}\vOnes > 0$ if there as at least one zealot. 
        We infer that $\mM_{\cP} = - (\mZ_{\cP} + \mL_{\cP})$ is strictly negative definite and 
                that it has full rank. 
        By the Implicit Function Theorem, there exists a unique function on an open interval around $\gamma = 0$ that satisfies properties (1) and (2).
        Property (3) follows because $\mM_{\cP}$ is
                negative definite on this interval.
    \end{proof}

    In \Cref{fig:small-gamma}, we illustrate \Cref{thm:small-gamma} when $\cG$ is the Zachary Karate Club graph \cite{zachary1977information}. 
    We show (a) the harmonic state, (b) an example steady state from the family that is described by \Cref{thm:small-gamma}, and (c) the complete family of these steady states as a function of $\gamma$.

    \begin{figure}
        \centering
        \includegraphics[width=\textwidth]{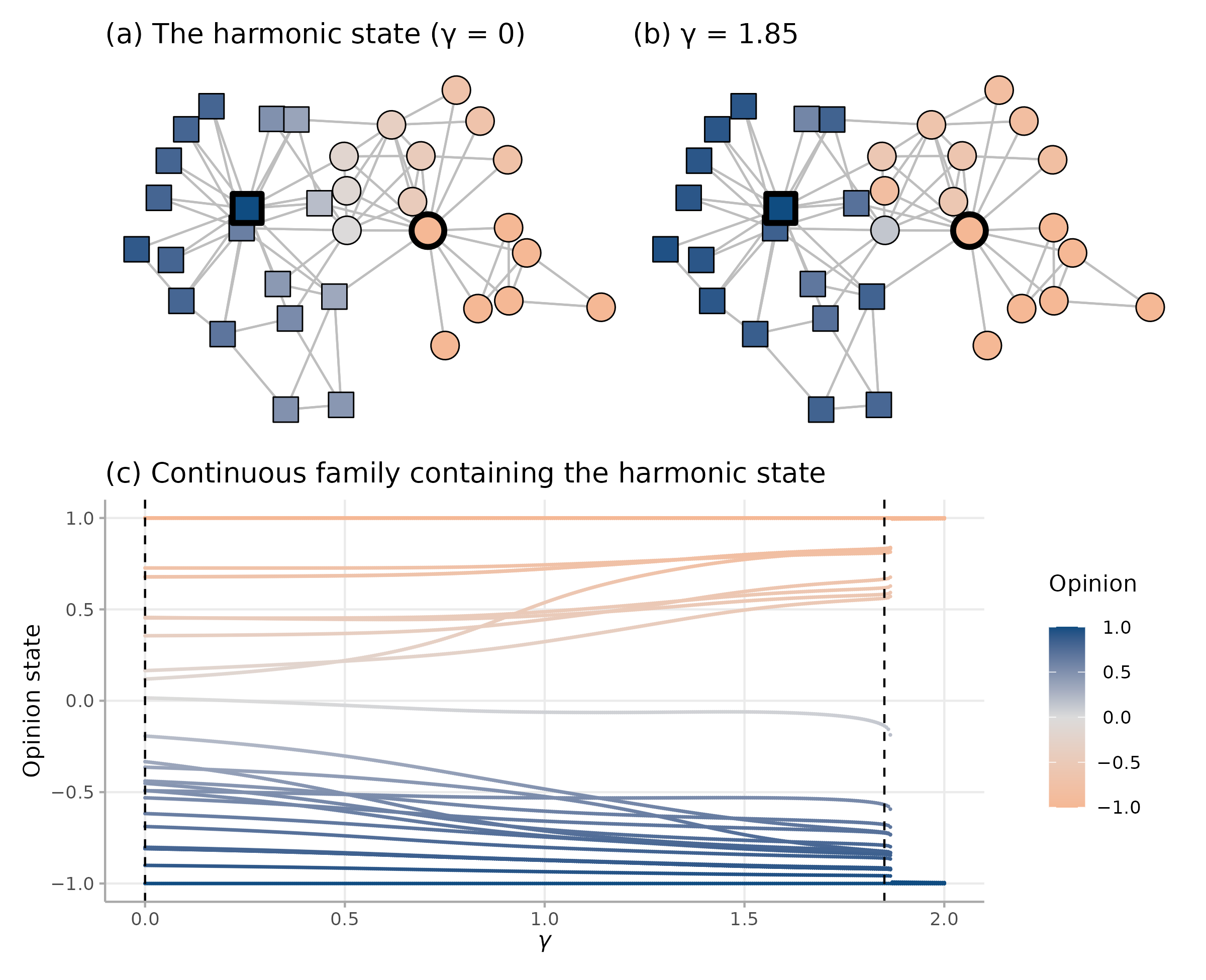}
        \caption{(Top) Two steady states of our SBCM in the family that is described by \Cref{thm:small-gamma} when \acc{$\cG$} is the Zachary Karate Club graph for the confidence bound $\delta = 0.5$ and different values of $\gamma$. 
        The colors correspond to opinion states. 
        We mark the two opposing zealots in the graph with thick black borders. 
        (Bottom) The complete family of steady states, with dashed lines that mark the values of $\gamma$ of
        the above graph visualizations. 
        For $\gamma \gtrapprox 1.9$, the Jacobian matrix is singular and the 
        guarantees
        of \Cref{thm:small-gamma} no longer hold.} \label{fig:small-gamma}
    \end{figure}

    \subsection{Large $\gamma$}
    
    In this subsection, we present two theorems about the behavior of equation \eqref{eq:dynamics} for large but finite values of $\gamma$.

    First, we use \Cref{prop:instability-condition} to bound the extent to which a node can separate from its neighbors in opinion space at a steady state.
    
    \begin{thm}\label{thm:isolation-bound}
        Let $\vx$ be a linearly stable steady state of equation \eqref{eq:dynamics}
        with parameters $\gamma$ and $\delta$. 
        For each node $i \in \cP$, there 
        exists $j \sim i$ such that 
        \begin{align} \label{eq:isolation-condition}
            \abs{x_i - x_j} \leq \sqrt{\max\{\delta, \gamma^{-1}\}}\,. 
        \end{align}
    \end{thm}

    \begin{proof}
        Fix a node $i \in \cP$. 
        Because $\vx$ is linearly stable, by \Cref{prop:instability-condition}, each node $i$ has a neighbor $j$ such that $(1 - w_{ij})(x_j - x_i)^2 \leq \frac{1}{2\gamma}$.
        Expanding this condition yields 
        \begin{align} \label{eq:isolation-bound}
            (x_j - x_i)^2 \leq \frac{1}{2\gamma} \paren*{1 + e^{-\gamma(x_i - x_j)^2 + \gamma \delta}} \,. 
        \end{align}
                If $(x_i - x_j)^2 \leq \delta$, then the bound \eqref{eq:isolation-condition} holds. 
                Now suppose that $(x_i - x_j)^2 > \delta$. 
        The exponent in the right-hand side of \eqref{eq:isolation-bound} is negative, so the 
                right-hand side is bounded above by $\gamma^{-1}$. 
        This again yields \eqref{eq:isolation-condition} and completes the proof. 
            \end{proof}

    \Cref{thm:isolation-bound} gives some insight into the structure of linearly stable steady states of equation \eqref{eq:dynamics} for large $\gamma$. 
            It is also natural to ask about the relationship between the steady states of our SBCM at large $\gamma$ and the steady states of the HK model. 
    Our next theorem shows that steady states of our
     SBCM converge to steady states of the HK model as $\gamma \rightarrow \infty$.

     Before proving
          the theorem, we establish some additional notation and discuss the convergence properties of $\mF$.
         Fix the confidence bound $\delta$. 
    We use the notation $\mF_{\gamma}$ to explicitly track the dependence of $\mF$ on the parameter $\gamma$.  
    Let $\cX(\gamma)$ denote the set of steady states $\vx$ of our SBCM with parameter $\gamma$.
       Let $\mF_\infty$ denote the pointwise limit $\lim_{\gamma \rightarrow \infty} \mF$, and let $\cX({\infty})$ denote the steady states of $\mF_\infty$. 
    Fix $\Delta > 0$ and $a > 0$, and define
        \begin{align*}
        \cS(a) \triangleq  \curlybrace*{\vx : \abs{\abs{x_i - x_j} - \sqrt{\delta}} \geq a \quad \text{for all} \quad i \sim j \quad \mathrm{and} \quad \norm{\vx}_\infty \leq \frac{\Delta}{2}}\;. 
    \end{align*}
    
    Finally, let $\hat{\cX}(a;\gamma) \triangleq \cS(a) \cap \cX(\gamma)$.

    \begin{lm} \label{lm:continuity}
        For any $a > 0$, the following statements hold:
        \begin{enumerate}
            \item The set $\cS(a)$ is complete. 
            \item As $\gamma \rightarrow \infty$, we have that $\mF \rightarrow \mF_\infty$ uniformly on $\cS(a)$. 
        \end{enumerate} 
    \end{lm}
     
    \begin{proof}
        The completeness of $\cS(a)$ follows from the fact that $\cS(a)$ is the intersection of a closed ball with a finite union of finite products of closed intervals.

        Recall that $w_{ij}^{(\infty)}$ is the 
                        step function in equation \eqref{eq:winfty}. 
        We define $s_{i}^{(\infty)} = \sum_{j \sim i} w^{(\infty)}_{ij}$. 
        To prove the uniform convergence of $\mF \rightarrow \mF_\infty$, fix $\epsilon > 0$ and $\vx \in \cS(a)$. 
        We calculate
        \begin{align*}
            \abs{f_{i}(\vx) - f_{\infty,i}(\vx)} = \sum_{j \sim i} (x_j - x_i)\sqbracket*{\frac{w_{ij}}{s_{i}} - \frac{w^{(\infty)}_{ij}}{s^{(\infty)}_{ij}}}\;.
        \end{align*}        
        Applying the Cauchy--Schwartz inequality yields 
        \begin{align}
            \abs{f_{i}(\vx) - f_{\infty,i}(\vx)} &\leq \sum_{j \sim i} \abs{x_j - x_i}\abs{\frac{w_{ij}}{s_{i}} - \frac{w^{(\infty)}_{ij}}{s^{(\infty)}_{ij}}} \nonumber \\ 
            &\leq \Delta\sum_{j \sim i} \abs{\frac{w_{ij}}{s_{i}} - \frac{w^{(\infty)}_{ij}}{s^{(\infty)}_{ij}}}\;,\label{eq:ineq-1}
        \end{align}
        where the second inequality in
        \eqref{eq:ineq-1} follows from the condition on $\cS(a)$ that $\norm{\vx}_\infty \leq \Delta$.
        We now show that $w_{ ij}\rightarrow w^{(\infty)}_{ij}$ uniformly on $\cS(a)$ as $\gamma \rightarrow \infty$. 
        Without loss of generality, we assume that $(x_j - x_i)^2 > \delta + a$;
        the case $(x_j - x_i)^2 < \delta - a$ is analogous.
        Because $(x_j - x_i)^2 > \delta + a$, we know that $w_{ij} \leq \frac{1}{1 + e^{\gamma a}}\rightarrow 0 = w^{(\infty)}_{ij}$. 
                For fixed $\epsilon > 0$, choosing $\Gamma = \frac{1}{a}\ln \left(\frac{1-\epsilon}{\epsilon}\right)$ implies that $\abs{w_{ij} - w^{(\infty)}_{ij}} < \epsilon$ for any $\vx$.
            
        We write 
        \begin{align*}
            \frac{w_{ ij}}{s_{i}} = \frac{1}{\sum_{j' \sim i}\frac{w_{ij'}}{w_{ij}}}\;.          
        \end{align*}
        The sum in the denominator converges uniformly and is bounded below by $1$. 
        We thus infer that $\frac{w_{ij}}{s_{i}}\rightarrow \frac{w^{(\infty)}_{ij}}{s^{(\infty)}_{ij}}$ uniformly. 
        Therefore, we can choose
                $\Gamma$ such that
        \begin{align*}
            \abs{\frac{w_{ij}}{s_{i}} - \frac{w^{(\infty)}_{ij}}{s^{(\infty)}_{ij}}} < \frac{\epsilon}{n^2\Delta} 
        \end{align*}
        for all $\gamma > \Gamma$. Inserting 
        this inequality into the bound \eqref{eq:ineq-1} yields 
        \begin{align*}
            \abs{f_{i}(\vx) - f_{\infty,i}(\vx)} < \frac{\epsilon}{n}\;,
        \end{align*}
        from which it follows that $\norm{\mF(\vx) - \mF_\infty(\vx)} < \epsilon$ for all $\vx \in \cS(a)$. 
        This completes the proof. 
    \end{proof}

    \begin{thm}\label{thm:HK-approx}
        Let $\{\gamma^{(\ell)}\}_\ell$ be a nondecreasing and unbounded sequence, and
        let $\cA(a)$ be the set of accumulation points of the set $\bigcup_\ell \hat{\cX}(a;{\gamma^{(\ell)}})$. 
        We then have that $\cA(a) \subseteq \hat{\cX}(a;\infty)$. 
    \end{thm}
   
    \begin{proof}
        We again
        use the notation $\mF_{\gamma}$ to explicitly track the dependence of $\mF$ on the parameter $\gamma$.  
        Let $\vx \in \cA(a)$ and fix $\epsilon > 0$. 
        By definition, there exists a sequence $\{\vx^{(\ell)}\}_\ell$ such that $\mF_{\gamma^{(\ell)}}(\vx^{(\ell)}) = \vzero$ and $\vx^{(\ell)} \rightarrow \vx$.

        We first compute 
        \begin{align}
            \norm{\mF_\infty(\vx)} &= \norm{\mF_\infty(\vx) - \mF_{\gamma^{(\ell)}}(\vx^{(\ell)})} \nonumber \\ 
            &\leq \underbrace{\norm{\mF_\infty(\vx) - \mF_{\infty}(\vx^{(\ell)})}}_{t_1} + \underbrace{\norm{\mF_\infty(\vx^{(\ell)}) - \mF_{\gamma^{(\ell)}}(\vx^{(\ell)})}}_{t_2} \label{eq:triangle-inequality}\;.
        \end{align}
        We now use \Cref{lm:continuity} and the fact that $\vx^{(\ell)} \in \hat{\cX}(a;\gamma^{(\ell)}) \subseteq \cS(a)$ for each $\ell$. 
        First, the completeness of $\cS(a)$ implies that $\vx \in \cS(a)$. 
        Because $\mF_{\gamma} \rightarrow \mF_{\infty}$ uniformly on $\cS(a)$, we know that $\mF_{\infty}$ is continuous on $\cS(a)$. 
        Therefore, we can select $\ell_1$ such that $t_1 < \epsilon/2$ for all $\ell > \ell_1$.
                The uniform convergence of $\mF_\gamma$ to $\mF_{\infty}$ on $\cS(a)$ implies that we can select $\ell_2$ such that $t_2 < \epsilon/2$ for all $\ell > \ell_2$.
                We choose $\bar{\ell} = \max(\ell_1, \ell_2)$; the inequality \eqref{eq:triangle-inequality} then implies that $\norm{\mF_{\infty}(\vx)} \leq \epsilon$.
        Because $\epsilon$ is arbitrary and in particular does not depend on $\ell$, we conclude that $\norm{\mF_{\infty}(\vx)} = 0$ and thus that $\mF_{\infty}(\vx) = \vzero$. 
        It follows that $\vx \in \hat{\cX}(a;\infty)$. This completes the proof.  
    \end{proof}

    \Cref{thm:HK-approx} asserts
        that the steady states of our SBCM that are bounded away from the manifolds $(x_i - x_j)^2 = \delta$ converge to steady states of the modified HK dynamics
        $\mF_\infty$. 
    We now show that a sequence of 
        linearly stable steady states of our SBCM must converge to
            a steady state of the unmodified HK model (which has the influence function \eqref{eq:HK-influence}).

    Let $\bar{\cX}(\gamma)$ denote the set of linearly stable steady states of equation \eqref{eq:dynamics} with parameter $\gamma$, and let 
    $\bar{\cA}$ be the set of accumulation points of the set $\bigcup_{\ell}\bar{\cX}(\gamma)$. 
 
    \begin{prop} \label{prop:stable-sequence}
        Suppose that $\vx$ has the property that $(x_i - x_j)^2  = \delta$ for some $i \sim j$. 
        We then have that $\vx \notin \bar{\cA}$. 
    \end{prop}

    \begin{proof}
        We prove this result by contradiction. Suppose 
                that $\vx \in \bar{\cA}$. 
                There exist sequences $\gamma^{(\ell)}\rightarrow \infty$ and $\vx^{(\ell)}\rightarrow \vx$ such that $\vx^{(\ell)}$ is a linearly stable steady state of equation \eqref{eq:dynamics}.
        Because $(x_i - x_j)^2 = \delta$, we have that 
        \begin{align*}
            \left(1 - w_{ij}(\vx^{(\ell)})\right)\left(x_j^{(\ell)}- x_i^{(\ell)}\right)^2 \rightarrow \frac{\delta}{2}\;. 
        \end{align*}
        However, $\frac{1}{2\gamma^{(\ell)}}\rightarrow 0$. 
        For sufficiently large $\ell$, we thus have 
        \begin{align*}
            (1 - w_{ij}(\vx^{(\ell)}))(x_j^{(\ell)} - x_i^{(\ell)})^2 > \frac{1}{2\gamma^{(\ell)}}\;, 
        \end{align*}
        which implies that $\vx^{(\ell)}$ is linearly unstable by \Cref{prop:instability-condition}. 
    \end{proof}
        
    \begin{remark}
        In concert, \Cref{thm:HK-approx,prop:stable-sequence} yield the following results:
        \begin{itemize}
            \item The steady states of 
            $\mF_{\gamma}$ 
            that are bounded away from the manifolds $(x_i - x_j)^2 = \delta$ approximate steady states of the modified HK dynamics $\mF_{\infty}$ as $\gamma\rightarrow \infty$. 
            \item Steady states of $\mF_\gamma$ that lie on one of the aforementioned manifolds are always unstable for sufficiently large $\gamma$. 
            In particular, the linearly stable steady states of our SBCM 
            approximate steady states of the unmodified HK dynamics (which has the influence function \eqref{eq:HK-influence}) as $\gamma \rightarrow \infty$.
              \end{itemize}
        
           \end{remark}

\section{Dynamics of our SBCM on special graph topologies} \label{sec:specialcases}
    In \Cref{sec:limitingprops}, we discussed the dynamics of our SBCM in the small-$\gamma$ regime (in which our SBCM resembles continuous-time averaging behavior) and the $\gamma\rightarrow \infty$ limit (in which our SBCM's steady states are related to those of a modified HK model). 
       We now study the behavior of the transition between these two extremes by exploring how the steady states and their stabilities change as we vary the model parameters $\gamma$ and $\delta$.
    
    The identification and characterization of the steady states of our SBCM for arbitrary graph topologies is a daunting task, and unfortunately we do not have a complete characterization of these states. 
    Instead of approaching the problem in complete generality, we study the dynamics of our SBCM on some special graph topologies whose structure allows us to make progress. 
    In particular, we examine the path graph (see \Cref{sec:path-graph}) and graphs with balanced exposure to zealots (see \Cref{sec:balanced-exposure}). 
    
    Narrowing our focus to special graph topologies limits
    our scope. 
    However, we are able to apply our
        results to certain types
        of graphs that we build
        from subgraphs with well-understood dynamics. 
      In \Cref{lm:disconnected}, we demonstrate that the dynamics of disconnected persuadable subgraphs 
        evolve independently. 
        In \Cref{thm:core}, we prove that a subgraph
        without zealots that is connected to rest of a graph
        via exactly one edge to a node $i$ has the same steady-state opinion as that of node $i$.
    In concert, these results help us understand the dynamics of our SBCM on any graph that we can partition into more easily-studied subgraphs.

    \begin{lm}\label{lm:disconnected}
        Suppose that the persuadable subgraph $\cP$ of a graph $\cG$ consists of two disconnected components, whose node sets are $I$ and $J$. 
        Let $\vx_I$ be the opinion vector restricted to nodes in $I$; we define $\vx_J$ analogously. 
        Let $\mF_I$ be the components of $\mF$ restricted to nodes $I$; we define $\mF_J$ analogously.
        It is then the case that
        $\frac{d\vx_I}{dt} = \mF_I(\vx_I)$ and $\frac{d\vx_J}{dt} = \mF_J(\vx_J)$.
    \end{lm}
    \begin{proof}
        We write 
        \begin{align}
            \frac{d\vx_I}{dt} = \mF_I(\vx) = \mF_I(\vx_I)\,,
        \end{align}
        where the second equality follows from the fact that $\mF_I$ has no nonzero terms with elements of $\vx_J$ because there is no edge between nodes in $I$ and nodes in $J$. 
        The same argument holds for $\mF_J$. 
    \end{proof}
    
    \begin{thm}\label{thm:core}
        Suppose that a graph $\cG$ is connected and that one can partition it into node sets $I$ and $J$ that satisfy the following properties: 
        \begin{enumerate}
            \item The set $J \subseteq \cP$. (That is, $J$ has no zealots.) 
            \item Any path from a node $j \in J$ to a zealot traverses
            the same node $i \in I$. 
        \end{enumerate}
        At a steady state of our SBCM, we then have that $x_j = x_i$ for all $j \in J$.  
    \end{thm}

    \begin{proof}
        We first show that $\vx_{J} = c\vOnes$ at a steady state of equation \eqref{eq:dynamics} for some constant $c$.
        We proceed by contradiction.
                To do this, we first suppose that $\vx_{J} \neq c\vOnes$ for any $c$. 
        We then have that $\vx_{J}$ has a largest entry (which need not be unique), which attains the value $\hat{x}$.
        Let $K$ be the set of nodes on which $\vx_{J}$ 
        attains the 
        value $\hat{x}$. 
        If $K = J$, then we may set $c = \hat{x}$. 
        If $K \neq J$, by the hypothesis that the graph $\cG$ is connected, there exist
        $k \in K$ and $\ell \in J \setminus K$ such that $k \sim \ell$. 
        We then write the right-hand side of equation \eqref{eq:dynamics} for node $k$ as 
        \begin{align}
            f_k(\vx) = \frac{1}{s_k}\sqbracket*{\sum_{k' \in K} w_{kk'}(x_{k'} - x_k) + \sum_{\ell \in J\setminus K} w_{k\ell}(x_{\ell} - x_k)}  \,. 
        \end{align}
        At \map{a} steady state, $f_k(\vx) = 0$. 
        However, this is impossible. 
        The first sum (which is over $K$) vanishes by construction and the second sum (which is over $J \setminus K$) is strictly negative. 
        Therefore, at steady state, there exists a constant $c$ such that $\vx_{J} = c\vOnes$. 

        We now argue that $c = x_i$. 
        The derivative of the opinion of node $j$ is 
                \begin{align}
            f_j(\vx) = \frac{1}{s_j}\sqbracket*{ w_{ij}(x_i - x_j) + \sum_{k \in J}w_{jk}(x_k - x_j)}\;. 
        \end{align}
        At steady state, the second term vanishes by our argument above. 
        Therefore, we 
        infer that 
        $x_i = x_j = c$. 
        This completes the proof. 
    \end{proof}

In \Cref{fig:diagram-subgraph}, we show an example to
    demonstrate the usefulness
        of \Cref{lm:disconnected} and \Cref{thm:core}.
    They allow us to completely characterize
    the dynamics of our SBCM on a graph with a somewhat complicated structure by decomposing it
    into smaller subgraphs with dynamics that are easier to understand. 
    We discuss two such situations in the next two subsections.

    \begin{figure}
        \centering
        \includegraphics[width=0.6\textwidth]{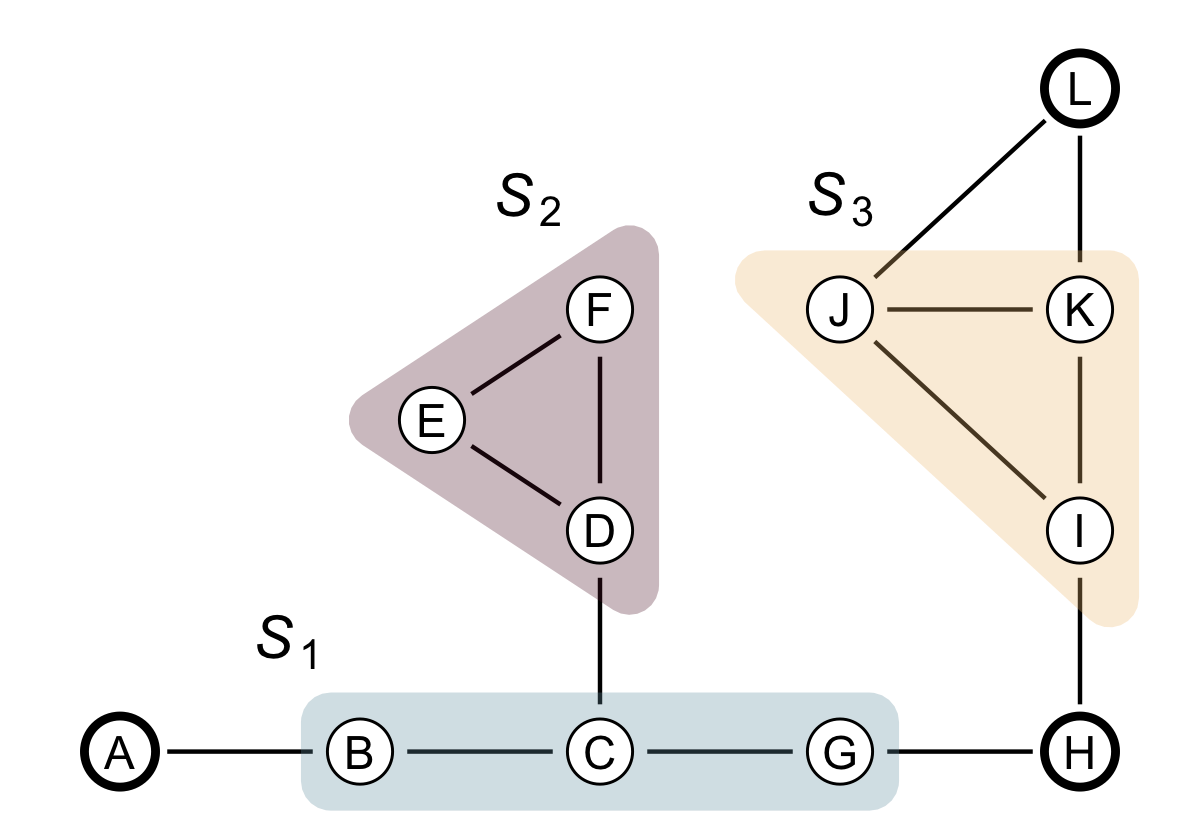}
        \caption{Illustration of \Cref{lm:disconnected} and \Cref{thm:core} 
            for a small graph with nine persuadable nodes and three zealots, which we highlight by thick borders.  
            The sets $S_1$ and $S_3$ are disconnected in the persuadable subgraph, and 
            they thus evolve independently of each other (by \Cref{lm:disconnected}) in our SBCM. 
            The nodes in the set $S_2$ satisfy the conditions on $J$ in \Cref{thm:core}. At steady state, $x_D = x_E = x_F = x_C$. 
        }
        \label{fig:diagram-subgraph}
    \end{figure}

    \subsection{Path graphs} \label{sec:path-graph}

    Let $P_n$ be a path graph with $n+2$ nodes, and suppose
    that the two extreme nodes of $P_n$ are zealots.
            The first zealot has index $1$ and opinion $x_1 = -(n+1)/2$.
    The opposing zealot has index $n+2$ and opinion $x_{n+2} = (n+1)/2$. 
    There are $n$ persuadable nodes, which are arranged in a path between the two symmetrically-placed zealots. 
    These persuadable nodes have indices $2, \ldots, n + 1$.

    When $\gamma = 0$, the {harmonic state} $\bar{\vx}$, which has components $\bar{x}_j = \frac{-(n+1)}{2} + j$, is the unique steady state of our SBCM;
    this state is linearly stable. 
        A direct calculation shows that $\bar{\vx}$ is a steady state of equation \eqref{eq:dynamics}
        on $P_n$ for any value of
        $\gamma$ but that its linear stability depends both on $\gamma$ and on $\delta$. 
        We seek to determine conditions that govern this dependence.

    \begin{lm}\label{lm:path-graph-eigs}
        Let 
        \begin{align} \label{eq:g}
            g(\gamma) \triangleq 2\gamma(1-v) - 1 > 0 \;,
        \end{align}
        where $v \triangleq \omega(1)$. 
        The harmonic state $\bar{\vx}$ of our SBCM on the path graph $P_n$ is linearly stable if and only if $g(\gamma) < 0$ and is linearly unstable if and only if $g(\gamma) > 0$.
    \end{lm}

    \begin{proof}
        As in our prior arguments, it suffices to consider the matrix $\mM_{\cP}$, rather than the full Jacobian matrix. 
        For this graph topology, the matrix $\mM_{\cP}$ has a simple form. 
        At the harmonic state $\bar{\vx}$, we have
                $(\bar{x}_j - \bar{x}_i)^2 = 1$ for all $i \sim j$, so
                $w(\bar{x}_i, \bar{x}_j) = \omega(1) = v$. 
        Let $s = vg(\gamma)$. 
        The matrix 
                                                                                                                                           \begin{align}
            \mM_\cP = s\begin{bmatrix}
                \phantom{-}2 & -1  & \phantom{-}0 & \cdots & \phantom{-}0 & \phantom{-}0 & \phantom{-}0  \\ 
                -1 & \phantom{-}2  & -1 & \cdots & \phantom{-}0 & \phantom{-}0 & \phantom{-}0 \\ 
                \phantom{-}0 & -1  & \phantom{-}2  & \ddots & \phantom{-}0 & \phantom{-}0 & \phantom{-}0 \\ 
                \phantom{-}\vdots & \phantom{-}\vdots & \ddots & \ddots & \ddots & \phantom{-}\vdots & \phantom{-}\vdots \\
                \phantom{-}0&  \phantom{-}0 &  \phantom{-}0 & \ddots & \phantom{-}2 & -1 & \phantom{-}0 \\ 
                \phantom{-}0&  \phantom{-}0 &  \phantom{-}0 & \cdots & -1 & \phantom{-}2 & -1 \\ 
                \phantom{-}0&  \phantom{-}0 &  \phantom{-}0 & \cdots & \phantom{-}0 & -1 & \phantom{-}2
            \end{bmatrix}
        \end{align}
        is both tridiagonal and Toeplitz. 
                We use known results for this type
        of matrix \cite{noschese2013tridiagonal} to explicitly write down the largest eigenvalue $\lambda$ of $\mM_\cP$. This eigenvalue is
                \begin{align}
            \lambda = 2s\sqbracket*{1 - \cos\paren*{\frac{\pi}{n+1}} }\,. 
        \end{align}
        The factor inside brackets is always positive, so the sign of $\lambda$ 
                matches the sign of $s$. 
        Because $v >0$, the sign of $s$ 
                matches the sign of $g(\gamma)$. 
        It follows that $\lambda < 0$ if and only if $g(\gamma) < 0$ and that $\lambda > 0$ if and only if $g(\gamma) > 0$. 
        This completes the proof.
    \end{proof}
    
      \begin{remark}

                For our SBCM on the path graph $P_n$ with symmetrically-placed zealots, \Cref{lm:path-graph-eigs} states
                        that the condition \eqref{eq:instability-condition} is both necessary and sufficient to ensure that $\bar{\vx}$ is linearly stable. 
    \end{remark}
    
    The following result gives the qualitative dependence of the linear stability of $\bar{\vx}$ on $\gamma$. 
    This dependence is controlled by the value of $\delta$. 
      Let $y$ be the unique negative real solution of the equation
      \begin{align}
    \label{eq:y-def}
        e^{y - 2} = \frac{y^2}{4}\;. 
    \end{align}
    Numerically, $y \approx -0.31$.

    \begin{thm}\label{thm:path-graph}
    For our SBCM on the 
            path graph $P_n$ with symmetrically-placed zealots, 
            the following statements hold:
    \begin{enumerate}
        \item If $\delta \in [0, 1]$, there exists a value $\gamma_c > 0$ such that the harmonic state $\bar{\vx}$ is linearly stable for all $\gamma \in [0, \gamma_c)$ and is linearly unstable for all $\gamma > \gamma_c$. Additionally, $\gamma_c = 1$ when $\delta = 1$.
              \item If $\delta \in (1, 1 - y)$, there exist $\gamma_1, \gamma_2 > 0$ such that the harmonic state is linearly unstable if and only if $\gamma \in (\gamma_1, \gamma_2)$ and is linearly stable if either $\gamma < \gamma_1$ or $\gamma > \gamma_2$.
        \item If $\delta > 1 - y$, the harmonic state $\bar{\vx}$ is linearly stable for all $\gamma > 0$.  
    \end{enumerate}
    \end{thm}

    \begin{proof}
        We proceed by examining the equation $g(\gamma) = 0$. 
        Dividing both sides of the equation by $v-1$ gives the equivalent equation $h(\gamma) = 0$, where 
        \begin{align*}
            h(\gamma) \triangleq \frac{1}{1-v} - 2\gamma = 1 + e^{-\gamma(1-\delta)} - 2\gamma\;.
        \end{align*}
        \Cref{lm:path-graph-eigs} implies that the harmonic state $\bar{\vx}$ is linearly stable on $P_n$ if and only if $h(\gamma) > 0$ and is linearly unstable if and only if $h(\gamma) < 0$. 
        We note two features of the function $h$. 
        First, $h$ is strictly convex as a function of $\gamma$ when $\delta \neq 1$, so the equation $h(\gamma) = 0$ must have either zero, one, or two solutions.
        Second, $h(0) = 2$. 
        
        \medskip

            \textbf{Case 1.} Suppose that $\delta \in [0,1)$. 
                        As $\gamma \rightarrow \infty$, the function $h(\gamma)$ does not have a lower bound.
            Because $h$ is continuous, the equation $h(\gamma) = 0$ must have at least one solution.
            Additionally, because $h$ is strictly convex, this equation can have at most two solutions. 
            However, because $h(0) = 2$ and $h(\gamma)$ does not have a lower bound, it
                        must cross the horizontal axis an odd number of times. 
            We conclude that $h(\gamma) = 0$ has a unique positive real solution, 
                        which we denote by $\gamma_c$. 
            Finally, if $\delta = 1$, then $h(\gamma) = 2 - 2\gamma$ and a direct computation gives the unique solution $\gamma_c = 1$. This completes the proof of this case.  

       \medskip

            \textbf{Case 2.} 
            Because $h$ is strictly convex, the condition
                         $\frac{\partial h(\bar{\gamma})}{\partial \gamma} = 0$ is both necessary and sufficient
                          for $\bar{\gamma}$ to be 
                          the unique global minimizer of $h(\gamma)$.
            When $\delta > 1$, the solution of  $\frac{\partial h(\bar{\gamma})}{\partial \gamma} = 0$ is
            \begin{align*}
                \bar{\gamma} = \frac{\ln 2 - \ln(\delta - 1)}{\delta - 1}\;. 
            \end{align*} 
                        Because $y \approx -0.31$, the condition $\delta \in (1, 1-y)$ ensures that $\bar{\gamma} > 0$. 
            The minimum of $h$ is then 
            \begin{align*}
                h(\bar{\gamma}) = 1 + \left[2 \times \frac{1 - \ln 2 + \ln(\delta - 1)}{\delta - 1}\right] \;. 
            \end{align*}
            We now check the sign of $h(\bar{\gamma})$. 
            The condition $h(\bar{\gamma}) < 0$ yields
                        \begin{align*}
                \frac{(1-\delta)^2}{4} < e^{-1 - \delta}\;,
            \end{align*}
            which has a solution if and only if $\delta < 1 - y$. 
            Therefore, provided that $1 < \delta < 1-y$, the minimum value $h(\bar{\gamma})$ is negative. 
            The strict convexity of $h$ implies that 
                         $h(\gamma) = 0$ has precisely two solutions, which we denote by $\gamma_1$ and $\gamma_2$. Both of these solutions are positive because $h(0) > 0$ and $h(\bar{\gamma}) < 0$.
            For $\gamma \in (\gamma_1, \gamma_2)$, we have $h(\gamma) < 0$ and the harmonic state $\bar{\vx}$ is linearly unstable. For $\gamma < \gamma_1 $ or $\gamma > \gamma_2$, we have $h(\bar{\gamma}) > 0$ and the harmonic state
                        $\bar{\vx}$ is linearly stable.  
            This completes the proof of this case.

       \medskip

        \textbf{Case 3.} If $\delta > 1 - y$, it follows that $h(\bar{\gamma}) > 0$. 
        Because $\bar{\gamma}$ is the global minimizer of $h$, we infer that $h(\gamma) > 0$ for all $\gamma$. 
        This implies that the harmonic state $\bar{\vx}$ is linearly stable for all $\gamma$. This completes the proof.
    \end{proof}

    To study the structure of steady states of equation \eqref{eq:dynamics}
        other than the harmonic state $\bar{\vx}$, we consider two one-dimensional families of states.
    For simplicity, we assume that $n$ (i.e., the number of nodes of $\cG$) is even for each of these families. 
    Both families of states have the form 
    \begin{align} \label{eq:path-graph-1d-family}
        \vx_\theta = (1 - \theta)\bar{\vx}  + \theta\frac{n+1}{2}\vv\;, 
    \end{align}
    where $\vv$ is an opinion vector and $\theta \in [0,1]$. We classify the families based on the structure of $\vv$. 
    Broadly speaking, the first family consists of ``polarization-like'' states; the parameter $\theta$ interpolates between the harmonic state and a symmetrically polarized state. 
    The second family consists of ``consensus-like'' states; the parameter $\theta$ interpolates between the harmonic state and consensus. 
    Inserting either of these families into equation \eqref{eq:dynamics} reduces the dynamics to a single dimension.
    
    In the first family of states,
        which we illustrate in \Cref{fig:path-graph}(a), $\vv$ is the vector $(-1, \ldots, -1, 1, \ldots, 1)$; there are $n/2 + 1$ copies each of the values $-1$ and $1$. 
        The condition $f_i(\vx_\theta) = 0$ is always satisfied for $i \in \{ 1,\ldots,n/2 - 1\}$ and $i \in \{ n/2 + 2,\ldots, n\}$. 
    For $i = n/2$ and $i = n/2 + 1$, the associated conditions $f_i(\vx_\theta) = 0$ 
        are the same,
                so we consider only the former. 
                           Let $m = n/2$. 
        Inserting \eqref{eq:path-graph-1d-family} into \eqref{eq:dynamics} yields 
        \begin{align}
        \frac{dx_m}{dt} = f_{m}(\vx_\theta) = \frac{(1-\theta )\omega(1 - \theta ) + (1 + \theta n)\omega(1 + \theta n)}{\omega(1-\theta ) + \omega(1 + \theta n)}\;,\label{eq:path-graph-1d-update}
    \end{align}
    which gives
        a steady state when
    \begin{align*}
        (1-\theta)\omega(1 - \theta) + (1 + \theta n)\omega(1 + \theta n) = 0\;.
    \end{align*}
    In \Cref{fig:path-graph}(d), we show the number of linearly stable steady states of the one-dimensional update equation \eqref{eq:path-graph-1d-update}.

    The dashed curve in \Cref{fig:path-graph}(d) is 
        $g(\gamma) = 0$. 
    As in \Cref{lm:path-graph-eigs,thm:path-graph}, this curve divides the $(\gamma,\delta)$ plane into regions in which the harmonic state
         is linearly stable and linearly unstable. 
    Below this curve, the harmonic state is linearly unstable. 
    If $\delta \leq 1$, a horizontal line 
        crosses the dashed line exactly once (at the value $\gamma_c$), as described in Case 1 of \Cref{thm:path-graph}. 
    If $1 < \delta < 1-y$, a horizontal line 
        crosses the dashed boundary twice, giving a destabilization and subsequent restabilization of the harmonic state
        as 
        $\gamma$ increases, as described in Case 2 of \Cref{thm:path-graph}. 
    Finally, if $\delta > 1 - y$, a horizontal line never crosses the dashed boundary and the harmonic state
        is linearly stable for all $\gamma$, as described by Case 3 of \Cref{thm:path-graph}. 
    
            In \Cref{fig:path-graph}(d), we also see an additional steady state (with 
        $\theta > 0$), which is unstable when $\gamma$ is close to $0$, stabilizes as $\gamma$ increases, and then destabilizes as $\gamma$ increases further.
    For very small values of $\gamma$, the harmonic state is the only linearly stable steady state of the form \eqref{eq:path-graph-1d-family}.
    As $\gamma$ increases, a second linearly stable state emerges; for this state, $\theta > 0$.

    In \Cref{fig:path-graph}(b,e), we consider a second family of states, which share the form
        \eqref{eq:path-graph-1d-family}. 
    This time, the entries of the vector $\vv$ are 
    \begin{align}
        v_i = \begin{cases}
        0\,, &\quad i \in \cP \\ 
        \bar{x}_i\,, &\quad \text{otherwise\,.}
        \end{cases} \label{eq:path-graph-1d-family-2}
    \end{align}
    For a narrow range of values of $\gamma$, there is a linearly stable steady state with $\theta > 0$. (See the yellow region of \Cref{fig:path-graph}(e).)
       In this state, the persuadable nodes are in approximate consensus; they are influenced more by each other's opinions than by the zealots. 
    As $\gamma$ increases, this steady state destabilizes, and then the harmonic state is the only remaining linearly stable {steady state}. 
    As before, the harmonic state is linearly unstable
    below the dashed curve and linearly stable above it.

    By comparing \Cref{fig:path-graph}(d) and \Cref{fig:path-graph}(e), we observe that the region of stability for the large-$\theta$ state is larger for the polarized state 
        \eqref{eq:path-graph-1d-family} than it is for the state
        \eqref{eq:path-graph-1d-family-2}. 
    This is a simple, interpretable way in which the path-graph topology favors polarization over consensus. 
    
    The families of states 
        that we described above do not include not all linearly stable steady states. 
        In \Cref{fig:path-graph}(c), we show an example of another linearly stable
        state.

    \begin{figure}
        \centering 
        \includegraphics[width=1\textwidth]{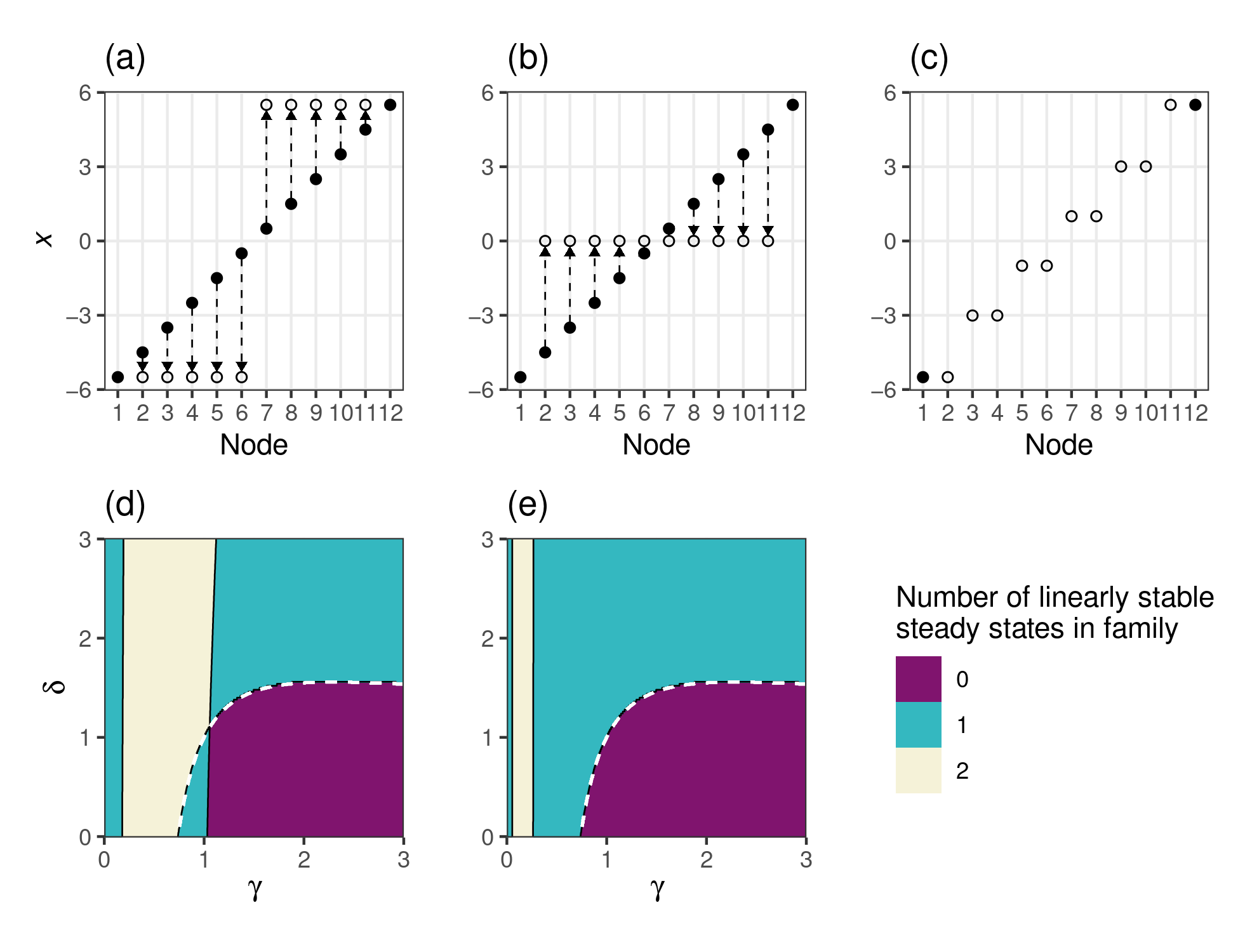}
        \caption{A pair of one-dimensional families of steady states of our SBCM in the 12-node path graph; these states take the form in
                 equation \eqref{eq:path-graph-1d-family}.
        (a) This family of steady states interpolates between the harmonic state (filled black disks) and a symmetrically polarized state (hollow disks), with $\vv = (-1,\ldots,-1,1,\ldots, 1)$; there are $n/2+1$ copies of each of the values $-1$ and $1$. 
        The dashed arrows correspond to changing the parameter $\theta$. 
        (b) This family of steady states interpolates between the harmonic state and a state in which all persuadable nodes are in consensus (hollow disks), with $\vv$ given by \eqref{eq:path-graph-1d-family-2}.        
                (c) An example of a steady state that is neither in the family in (a) nor in the family in (b). 
        We obtain this state numerically using the parameter values $\delta = 0.01$ and $\gamma = 60$. 
        (d) The number of linearly stable steady states in the family in (a) as a function of the parameters $\gamma$ and $\delta$. 
        (e) The number of linearly stable steady states in the family in (b) as a function of the parameters $\gamma$ and $\delta$.
                The dashed white curve is the curve $g(\gamma) = 0$, which (by \Cref{lm:path-graph-eigs}) separates the regions in which the harmonic state is linearly stable from those in which it is linearly unstable.}   
        \label{fig:path-graph}
    \end{figure}

    \subsection{Graphs with balanced exposure} 
    \label{sec:balanced-exposure}

    We now seek to leverage symmetry while moving beyond the particular structure that is imposed by a path graph. To explore this idea, we introduce the notion of \emph{balanced exposure.}  
    We say that a graph with two zealots satisfies the \emph{balanced-exposure} (BE) condition if no persuadable node is adjacent to exactly one zealot. 
    Throughout this subsection, we assume that the opposing zealots have opinions $-1$ and $1$.
    In a BE graph, it is possible to fully characterize the linear stability of the harmonic state.

    \begin{thm} \label{thm:BE-harmonic}
        Let the graph $\cG$ satisfy BE. The following statements hold\map{:}
        \begin{enumerate}
            \item For any $\gamma$, the harmonic state $\bar{\vx} = \vzero$ is a steady state of equation \eqref{eq:dynamics}.
            \item The harmonic state $\bar{\vx}$ is linearly stable if and only if $g(\gamma) > 0$. 
        \end{enumerate}
    \end{thm}

    \begin{proof} 
        Let $u = \omega(0)$ and $v = \omega(1)$.
        Fix a node $i \in \cN$. 
        If $i$ is not adjacent to any
        zealots, then $f_i(\bar{\vx}) = 0$ because $\bar{x}_i = \bar{x}_j = 0$ for all $j \sim i$. 
        If $i$ is adjacent
        to two zealots, then
        \begin{align*}
            f_i(\bar{\vx}) =
                \frac{1}{d_iu + 2v} \paren*{(1-0)v + (0-1)v} = 0\;,
        \end{align*}
        where $d_i$ is the degree of node $i$. Therefore, for all $i$, we have that $f_i(\bar{\vx}) = 0$ and hence that the harmonic state $\bar{\vx} = \vzero$ is a steady state of equation \eqref{eq:dynamics}.

        We now study the linear stability of the steady state $\bar{\vx}$ by examining the spectrum of the matrix $\mM_\cP$. 
        Let $\mA$ be the adjacency matrix of the persuadable subgraph, let $\mD$ be the diagonal matrix of degrees in the persuadable subgraph, and let $\mK$ be the diagonal matrix of zealot-degrees (i.e., the number of zealots that are attached to each node). 
               With this notation, we write
        \begin{align}
            \mM_\cP &= u(\mA - \mD) - v(1-2\gamma(1-v))\mK \nonumber\\ 
                &= -u\mL - v(1-2\gamma(1-v))\mK\;, \label{eq:M}
        \end{align}
        where $\mL$ is the combinatorial graph Laplacian of $\cG$. 
        For any 
                unit vector $\vv$, we have 
        \begin{align*}
            \vv^T\mM_\cP\vv = -u\vv^T\mL\vv - v(1 - 2\gamma(1-v))\vv^T\mK\vv\;. 
        \end{align*}
        Because $u > 0$ and $\mL$ is 
                positive semidefinite, the first term is nonpositive and it is $0$ only if $\vv = \frac{1}{\sqrt{n}}\vOnes$.
        Furthermore, $\mK$ is diagonal with nonnegative entries, which implies that $\vv^T\mK\vv \geq 0$ and $\vOnes^T\mK\vOnes > 0$.  
        Suppose that $1 - 2\gamma(1-v)>0$. 
        It follows that $\mM_\cP$ is negative definite because it is not possible for $\vv^T\mL\vv$ and $\vv^T\mK\vv$ to vanish simultaneously. 
        Suppose instead that $1 - 2\gamma(1-v)\leq 0$. 
        We then have that $\vOnes^T\mM_\cP\vOnes \geq 0$, so $\mJ$ has a nonnegative eigenvalue. 
        This completes the proof of the second statement of the theorem.     
    \end{proof}
    
    \begin{remark}
        Because the second statement of \Cref{thm:BE-harmonic} is the same linear-stability criterion as in \Cref{thm:path-graph}, the conclusions of \Cref{thm:path-graph} about the linear stability of the {harmonic state} hold for BE graphs. 
    \end{remark}

    One natural intuition for 
        opinion dynamics on a graph is that the stability of a state or the amplification of a perturbation depends on its relationship to the topology of the 
        graph. 
        For example, consider a perturbation $\vv$ in which $n/2$ entries have value $+1$ and $n/2$ entries have value $-1$. 
    This perturbation splits persuadable nodes into two equal-sized groups. 
        It is natural to conjecture that this perturbation is amplified more strongly if these groups are
                communities of a
                graph.\footnote{In idealized form, the ``communities'' of a graph
           are densely connected sets of nodes that are connected sparsely to other dense sets of nodes~\cite{porter09}.}  
    In this situation, we say that the perturbation is \emph{aligned} with a graph's community structure.
        By contrast, a perturbation that splits the set of nodes into groups that are unrelated to a graph's community structure is \emph{unaligned} with that community structure. 
            Intuitively, perturbations that are aligned with 
        graph community structure are amplified more than unaligned perturbations. 
    The following result makes this intuition precise for a special subset of BE graphs.

    \begin{thm} \label{thm:BE-laplacian}
        Consider a graph with 
        two zealots. Suppose that the persuadable subgraph $\cG_\cP$ is $d$-regular (i.e., all nodes have degree $d$) for some $d$ and that every persuadable node is adjacent to both zealots. 
        We then have that the space of unstable directions at the harmonic state $\bar{\vx}$ is spanned by the eigenvectors $\{\vv_i\}$ of $\mL_\cP$ whose associated eigenvalues $\lambda_i$ satisfy
        \begin{align} \label{eq:BE-laplacian-eig-condition}
            \lambda_i \leq -\frac{2v(1-2\gamma(1-v))}{u}\;.
        \end{align} 
    \end{thm}
    
    \begin{proof}
            The hypothesis that $G_{\cP}$         is $d$-regular implies that $\mD = d\,\mI$, and the hypothesis that every persuadable node is adjacent to both zealots implies that $\mK = 2\mI$. 
               The Jacobian matrix 
                of the system \eqref{eq:dynamics} restricted to the persuadable subgraph is thus
        \begin{align*}
            \mJ_\cP = \mS_\cP^{-1} \mM_\cP = -\frac{1}{ud+2v}\mM_\cP\;.
        \end{align*}
        The space of unstable directions of $\mJ_\cP$ coincides with the space that is spanned by the eigenvectors of $\mM_\cP$ with nonnegative eigenvalues.  
        From equation \eqref{eq:M}, we observe that if $\vv$ is an eigenvector of $\mM_\cP$ with eigenvalue $\nu$, then $\vv$ must also be an eigenvector of $\mL$ with eigenvalue $\lambda = -\frac{\nu + 2v(1-2\gamma(1-v))}{u}$. 
        Requiring $\nu \geq 0$ completes the proof. 
    \end{proof}

    \begin{remark}
        When the persuadable subgraph $\cG_\cP$ is connected, there is a unique smallest eigenvalue $\lambda_1 = 0$ with corresponding eigenvector $\vOnes$. 
        This eigenvector corresponds to a uniform shift of all agents in the same direction in opinion space. 
        The emergence of this unstable direction is a ``consensus bifurcation'' (in the 
        terminology of \citet{franciBreakingIndecisionMultiagent2022}).
        Subsequent bifurcations as $\gamma$ increases can induce dissensus or polarization\map{.}
               \end{remark}

    In the limit $\gamma \rightarrow \infty$, the structure of the space of unstable directions depends on $\delta$.

    \begin{cor}
       Consider a graph $\cG$ that satisfies the hypotheses of \Cref{thm:BE-laplacian}. 
       In the limit $\gamma \rightarrow \infty$, the following statements hold\map{:}
        \begin{enumerate}
            \item If $\delta > 1$, the harmonic steady state $\bar{\vx}$ is linearly stable. 
            \item If $\delta = 1$, the space of unstable directions at $\bar{\vx}$ is the range of $\mL_\cP$. 
            \item If $\delta \in [0,1)$, the space of unstable directions at $\bar{\vx}$ is spanned by the vector $\vOnes$. 
        \end{enumerate}
    \end{cor}

      \begin{proof}
        For any $\delta \geq 0$, we have that $u = \omega(0) \rightarrow 1$ as $\gamma \rightarrow \infty$. 
        
        Suppose first that $\delta > 1$. 
        In this case, $v = \omega(1) \rightarrow 1$ exponentially fast as $\gamma \rightarrow \infty$. 
        Consequently, the right-hand side of the bound \eqref{eq:BE-laplacian-eig-condition} approaches the value $-2$. 
        Because $\mL_\cP$ is positive semidefinite, no eigenvalues satisfy \eqref{eq:BE-laplacian-eig-condition} in this limit. \Cref{thm:BE-laplacian} implies that $\bar{\vx}$ is linearly stable. 
        
        Now suppose that $\delta = 1$. 
        In this case, $v = {1}/{2}$ and the bound \eqref{eq:BE-laplacian-eig-condition} simplifies to 
        \begin{align}\label{eq:small-eigenvalue}
            \lambda_i \leq \frac{\gamma - 1}{u}\;. 
        \end{align}
        As $\gamma \rightarrow \infty$, every eigenvalue of $\mL$ satisfies 
        the bound \eqref{eq:small-eigenvalue}.
        Therefore, all eigenvectors of $\mL$ are present in the space of unstable directions. 
    
        Now suppose that $\delta \in [0,1)$.
        In this case, $v \rightarrow 0$ exponentially fast.
        Consequently, the right-hand side of the bound \eqref{eq:small-eigenvalue} becomes arbitrarily small. 
        The only eigenvalue of $\mL_\cP$ that satisfies \eqref{eq:small-eigenvalue} is $\lambda_1 = 0$. 
        Therefore, the only unstable direction is spanned by $\vOnes$. 
    \end{proof}

    \begin{figure}
    \centering
    \includegraphics[width=\textwidth]{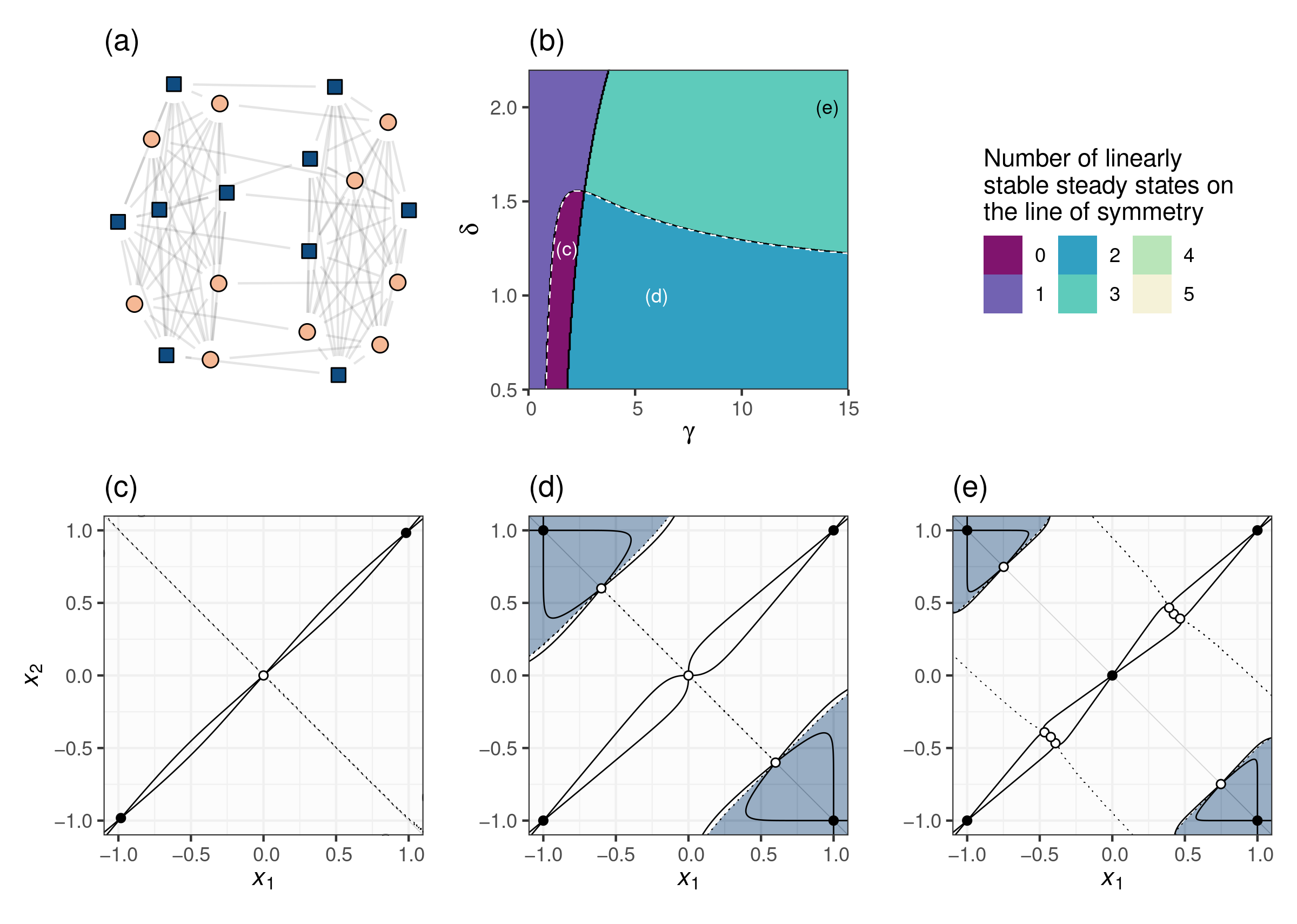}
    \caption{
        (a) A visualization of the persuadable subgraph of a graph; shape and color distinguish the classes $\cP_1$ and $\cP_2$ of persuadable nodes. The classes $\cP_1$ and $\cP_2$ are not aligned with the graph's community structure.
        There is one zealot with opinion $-1$ and one zealot with opinion $1$;
        each zealot is adjacent
        to every persuadable node (not shown). 
        (b) The number of linearly stable steady states on
                the line $x_1 = -x_2$ as a function of the parameters $\delta$ and $\gamma$. 
        The dashed white curve corresponds to the condition in \Cref{thm:BE-harmonic} for the linear stability of the {harmonic state} $\bar{\vx}$. 
        We mark the regions of parameter space that correspond to the 
        phase portraits in panels (c)--(e).
        (c)--(e) Phase portraits in the variables $x_1$ and $x_2$ for three different combinations of $\delta$ and $\gamma$. 
        The solid gray diagonal line is $x_1 = -x_2$ (i.e., the line of symmetry); polarized states on this line are symmetrically polarized, whereas polarized states that are not on this line are asymmetrically polarized. 
        The solid black disks are linearly stable steady states, and the hollow disks are linearly unstable steady states. 
        The solid black curves are nullclines. 
        We shade the regions of attraction based on the value of $\abs{x_1 - x_2}$ at the associated attractor; we use darker shades in regions with more polarized behavior as $t \rightarrow \infty$.
        } \label{fig:paired-cliques-misaligned}
    \end{figure}

    \begin{figure}
        \centering
        \includegraphics[width=\textwidth]{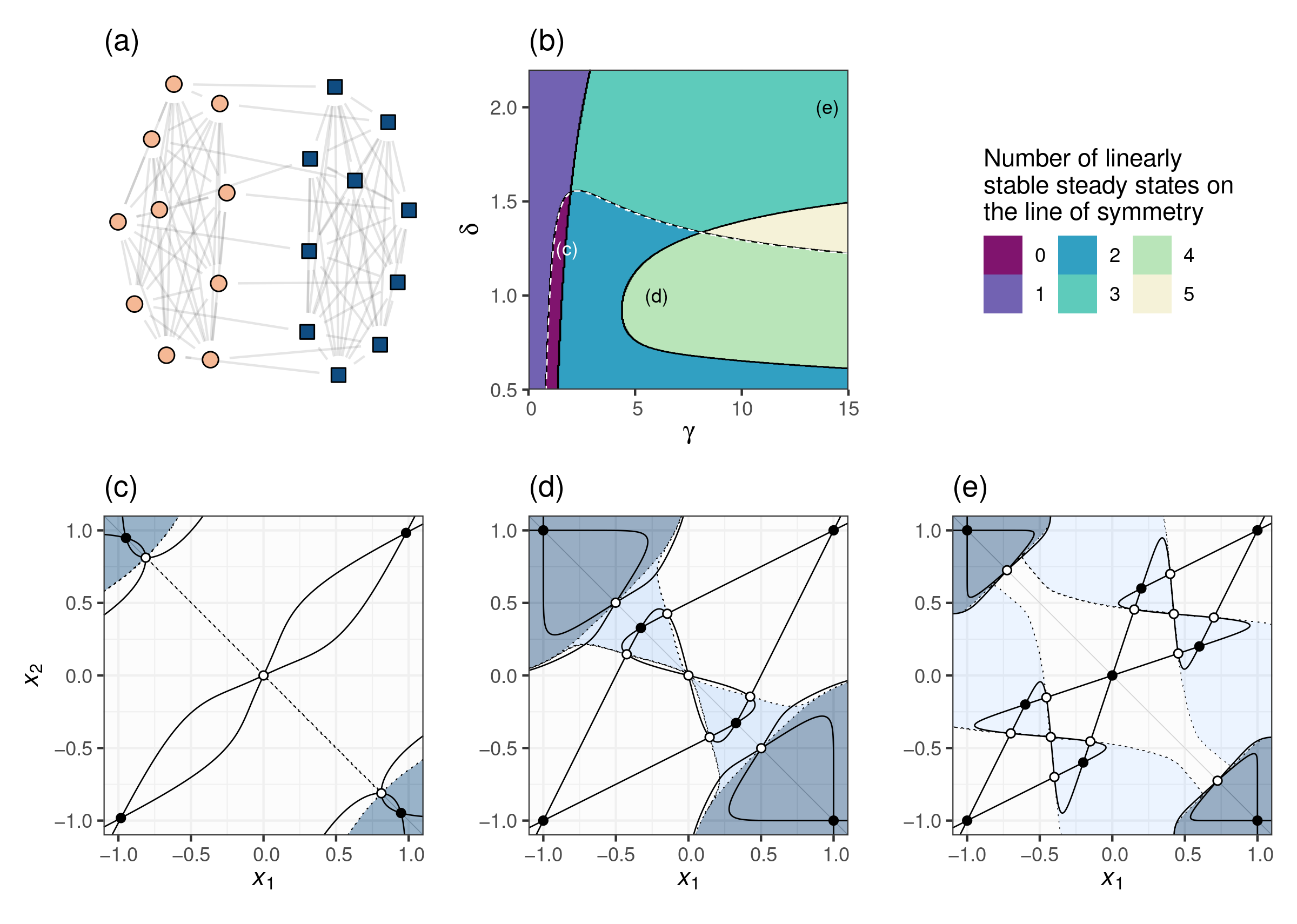}
        \caption{(a) A visualization of the persuadable subgraph, (b) the number of linearly stable steady states 
                on the line $x_1 = -x_2$ (i.e., the line of symmetry) as a function of the parameters $\delta$ and $\gamma$, and (c)--(e) phase portraits in the variables $x_1$ and $x_2$ for three different combinations of $\delta$ and $\gamma$ for the classes $\cP_1$ and $\cP_2$ that align with the community structure of the persuadable subgraph. 
        The parameter values that we use panels (c)--(e) are the same as those that we used in \Cref{fig:paired-cliques-misaligned}(c)--(e).    
                                        } \label{fig:paired-cliques-aligned}
    \end{figure}

    \Cref{thm:BE-laplacian} states that, as $\gamma \rightarrow \infty$, the directions in which the harmonic state destabilizes first are the directions with the large projections onto the eigenspace of $\mL_\cP$ that is associated with its smallest eigenvalues.\footnote{Recall that $\mL_\cP$ is the combinatorial graph Laplacian of the persuadable subgraph. Therefore, its eigenvalues are real and nonnegative.}
    The direction that emerges first is $\vv_1 = \vOnes$, which corresponds to all persuadable nodes shifting their opinions towards that of
    a single zealot. The next direction to emerge is the Fiedler eigenvector $\vv_2$. It is common to use the signs of the entries of the Fiedler eigenvector for
        spectral community detection in networks when seeking
        two communities \citep{von2007tutorial}. One can use additional eigenvectors to identify larger numbers of (finer-grained) communities.

    In \Cref{fig:paired-cliques-misaligned,fig:paired-cliques-aligned}, we explore the interplay between graph structure and linear stability in a BE graph with two communities. 
    Each community is a 10-node clique, and each
    node in a clique is adjacent
        to exactly one node in the other clique. 
    Additionally, all of these nodes (which are persuadable)
    are adjacent to both of two opposing zealots. 
    The zealots are not adjacent to each other. 
    In this graph, the signs of the Fiedler eigenvector $\vv_2$ distinguish the two cliques. 
    Therefore, we can measure the alignment of any perturbation 
        with the community structure using the projection of that perturbation onto $\vv_2$.

    We consider states in which we can partition the persuadable subgraph into two equal-sized classes $\cP_1$ and $\cP_2$ such that all nodes in $\cP_1$ have the same opinion $x_1$ and all nodes in $\cP_2$ have the same opinion $x_2$. 
    This partition reduces our system to the two variables $x_1$ and $x_2$, allowing us to visualize it in two-dimensional phase portraits.
        These states emerge as $\gamma$ increases via ``polarized dissensus bifurcations'' (in the terminology of \citet{franciBreakingIndecisionMultiagent2022}). 
            In \Cref{fig:paired-cliques-misaligned}, we consider a configuration in which each node in both $\cP_1$ and $\cP_2$ has exactly five neighbors in each of the two classes.
        When $x_1 = -x_2$, this configuration is aligned with one of the eigenvectors that corresponds to
    the third-smallest eigenvalue $\lambda_3$ of $\mL_\cP$.
    Every such eigenvector is orthogonal to $\vv_2$, so this configuration is
    unaligned with the graph's community structure. 
    In \Cref{fig:paired-cliques-aligned}, we consider a configuration in which each node in $\cP_1$ has nine neighbors in $\cP_1$ and one neighbor in $\cP_2$.
        When $x_1 = -x_2$, this configuration is aligned with
        the Fiedler vector of the graph.
        In this sense, it is aligned with the graph's community structure.

    To give some qualitative guidance about
        the dependence of the system on the parameters $\delta$ and $\gamma$, we count the number of linearly stable steady states that lie on the line $x_1 = -x_2$ in panel (b) of \Cref{fig:paired-cliques-misaligned} and \Cref{fig:paired-cliques-aligned}. 
    Steady states on this line correspond to symmetric polarization, in which the opinions of the nodes of each class
    are equidistant from the origin. 
    This analysis does not capture asymmetrically polarized states, in which the nodes of one class possesses a more extreme opinion than those of the other; below we will see examples of such states.

   In the unaligned configuration in \Cref{fig:paired-cliques-misaligned}, there are four possible numbers of steady states.
           For very small values of $\gamma$, only the harmonic state is linearly stable. 
    For $\delta < 1-y$, where $y$ satisfies \Cref{eq:y-def}, increasing $\gamma$ causes the {harmonic state} to linearly destabilize. Consequently, there are
        no linearly stable steady states on the line $x_1 = -x_2$ in panel (c).
        Increasing $\gamma$ further generates linearly stable polarized steady states on the line $x_1 = -x_2$. 
    Depending on the value of $\delta$, it is possible for a linearly stable harmonic state to accompany these steady states; see panels (d) and (e).

    When a partition into opinion classes aligns with a graph's community structure, we observe richer behavior 
    (see \Cref{fig:paired-cliques-aligned}) than in the above unaligned situation. 
    Depending on the values of $\delta$ and $\gamma$, there are five different possible numbers of linearly stable steady states on the line $x_1 = -x_2$.
    Our choice to restrict attention to this one-dimensional space reflects computational limitations that prevent us from enumerating all stable states in the two-dimensional phase 
        space for many combinations of $\delta$ and $\gamma$.  
    In panel (c), we highlight that the alignment with graph structure encourages polarization. On the line $x_1 = -x_2$, we observe symmetric, highly polarized steady states with $\abs{x_1 - x_2} \approx 2$.
    By contrast, for the unaligned example in \Cref{fig:paired-cliques-misaligned}, there are
        no linearly stable steady states for this parameter combination on the line $x_1 = -x_2$. 
        In \Cref{fig:paired-cliques-aligned}(d),
        we show a parameter combination with four linearly stable steady states on the line $x_1 = x_2$; 
        these include two highly polarized states with $\abs{x_1 - x_2} \approx 2$ and two moderately polarized states with $\abs{x_1 - x_2} \approx 0.6$. 
    In panel (e), we see that the the moderately polarized states have moved off of the line $x_1 = -x_2$; this asymmetric polarization is reminiscent of the ``moderate--extremist disagreement'' of \citet{franci2019model}. 
    Additionally, the {harmonic state} is again linearly stable.
    By comparing panels (c)--(e) in \Cref{fig:paired-cliques-misaligned,fig:paired-cliques-aligned}, we see
    that the combined volume of the attraction basins of
    consensus states on the line $x_1 = x_2$ is smaller in \Cref{fig:paired-cliques-aligned}, indicating a greater propensity towards enduring disagreement from uniformly random initial opinions.

\section{Conclusions and discussion} 
\label{sec:conclusions}
We studied a sigmoidal bounded-confidence model (SBCM), which interpolates smoothly
between averaging dynamics and 
\linebreak
bounded-confidence dynamics, and used it to examine opinion dynamics on networks.
  We showed that its long-term dynamics
  are related to the long-term behaviors of 
the averaging and bounded-confidence dynamics in the associated limits. 
We also performed linear stability analysis of our SBCM's steady states for certain graph topologies.  
We thereby obtained qualitative descriptions of how bounded-confidence
behavior emerges from averaging behavior 
as the sigmoidal opinion-updating function's
steepness parameter $\gamma \rightarrow \infty$. 
This yielded both analytical and computational insights into the
relationship between graph topology and the stability of polarized opinion states. By considering special graph topologies --- first path graphs and then balanced-exposure graphs with community structure --- we were able to probe deeper into specific situations of interest.

Our work invites many further developments. 
For example, there remain 
fundamental model properties to analyze. 
One important question is when it is possible to approximate a steady state of an HK model by sequences of steady states of our SBCM as $\gamma \rightarrow \infty$. 
This question complements our result in \Cref{thm:HK-approx}. 
We offer the following conjecture. 
\begin{conj}
    Let $\vx$ be a steady state of an HK model with confidence bound $\sqrt{\delta}$, and let $C_\delta(\vx) \subset \R^n$ be the set of opinion vectors $\vy$ such that
    \begin{align*}
        (y_i - y_j)^2 \leq \delta \iff (x_i - x_j)^2 \leq \delta \quad \text{for all} \quad i \sim j\;. 
    \end{align*}
    There exist $\vx' \in C_{\delta}(\vx)$, a sequence $\curlybrace*{\gamma^{(\ell)}}_\ell$, and a sequence $\curlybrace*{\vx^{(\ell)}}_\ell$ such that 
    \linebreak
    $\mF_{\gamma^{(\ell)}}(\vx^{(\ell)}) = \vzero$ and $\vx^{(\ell)}\rightarrow \vx'$ as $\gamma \rightarrow \infty$. 
\end{conj} 
The set $C_\delta(\vx)$ includes all opinion vectors in which the same pairs of nodes as those in the vector $\vx$ are able to influence each other (in an HK model with confidence bound
$\delta$). Our conjecture states that every such pattern of mutual influence has a representative opinion vector that one can approximate by a sequence of steady states of our
SBCM. Additionally, although we focused in the present paper on
the structure of steady states of our SBCM, it seems worthwhile to study the dependence of the transient behavior of our SBCM on 
graph topology (perhaps using methods that are similar to those of \citet{xing2022transient}). 
Linear stability analysis of our SBCM does not allow one to determine its transient behaviors, and other asymptotic approaches may be helpful to describe them.

There are several other interesting ways to build on our work.
It is particularly desirable to analytically investigate more general graph topologies than the ones that we studied in \Cref{sec:specialcases}. 
There are also several possible modifications of the underlying model dynamics. 
One possibility is the incorporation of noise into the opinion-update rule \eqref{eq:dynamics} and studying the resulting stochastic differential equation (SDE). 
SDE models of opinion dynamics are less common than discrete-time stochastic and continuous-time deterministic opinion models, but some tractable models do exist \citep{de2016learning,liang2018continuous}. 
A particularly attractive benefit of incorporating noise into the opinion updates of an SBCM is that it may enable the development of methods to fit the ensuing models to experimental and observational data. 
Another possibility is to allow the parameters $\gamma$ or $\delta$ to vary stochastically with time and to study the resulting distribution of steady states.
Other promising extensions include the incorporation of multiple opinion dimensions \cite{brooks2020model,de2022multi}, contrarian agents (see \cite{juul2019hipsters} and references therein), and more general 
influence functions $w$. 

In interpreting our results about graph topology and the stability of polarized steady states, it is important to remember that our SBCM (like all other opinion models) is very limited as an empirical description of the dynamics of real-world political polarization. One important limitation is symmetry. 
Our findings treat opposing groups as behaving identically, but this is typically unrealistic. 
In particular, recent efforts suggest that this assumption appears to be a poor description of rising polarization in United States politics both for political elites \citep{leonard2021nonlinear} and for individual voters \citep{waller2021quantifying}. It is also worthwhile to study SBCMs that incorporate asymmetries in media influence, social-network structure, and behaviors in subpopulations of nodes.

\section*{Acknowledgements}
    We are grateful to Solomon Valore-Caplan for useful conversations, to
    Daniel Cooney for suggesting the extension of our SBCM to stochastic differential equations (that we mentioned briefly in our discussion of future work), and to two anonymous referees for 
    helpful comments. HZB was funded in part by the National Science Foundation (grant number DMS-2109239) through their program on Applied Mathematics.
    MAP was funded in part by the National Science Foundation (grant number 1922952) through their program on Algorithms for Threat Detection.
    Much of PSC's work was completed during his time at 
    University of California, Los Angeles.

\appendix

\section{Software}\label{software}
    Software that is sufficient to reproduce the computational experiments in our paper is available at
    \url{https://gitlab.com/philchodrow/sigmoidal-bounded-confidence}. 
    We performed our primary computations using the Julia programming language \citep{bezanson2017julia}, and we constructed visualizations using the {\sc ggplot2} package \citep{wickhamGgplot2ElegantGraphics2016} for the {\sc R} programming language \citep{rcoreteamLanguageEnvironmentStatistical2022}.

    \bibliographystyle{abbrvnat}
    
    \bibliography{refs08}

\end{document}